\documentclass[reqno, 11pt]{amsart}
\usepackage{amssymb}
\usepackage[usenames, dvipsnames]{color}
\usepackage{todonotes}
\usepackage{hyperref}
\usepackage{xcolor}
\usepackage{verbatim}
\usepackage{tikz,graphicx}
\usepackage{accents} 
\usepackage{enumitem}

\newtheorem{theorem}{Theorem}[section]
\newtheorem{corollary}[theorem]{Corollary}
\newtheorem{lemma}[theorem]{Lemma}
\newtheorem{proposition}[theorem]{Proposition}

\theoremstyle{definition}
\newtheorem{remark}[theorem]{Remark}

\theoremstyle{definition}
\newtheorem{definition}[theorem]{Definition}

\theoremstyle{definition}
\newtheorem{assumption}[theorem]{Assumption}

\makeatletter
\def\dashint{\operatorname%
{\,\,\text{\bf--}\kern-.98em\DOTSI\intop\ilimits@\!\!}}
\makeatother

\def\.5{\frac{1}{2}}

\def\bR{\mathbb{R}}

\def\bH{\mathbb{H}}

\def\cA{\mathcal{A}}
\def\cB{\mathcal{B}}
\def\cC{\mathcal{C}}
\def\cD{\mathcal{D}}

\def\cH{\mathcal{H}}

\def\cM{\mathcal{M}}

\def\cO{\mathcal{O}}
\def\cP{\mathcal{P}}
\def\cQ{\mathcal{Q}}

\def\cS{\mathcal{S}}

\def\cU{\mathcal{U}}
\def\cV{\mathcal{V}}
\def\cW{\mathcal{W}}

\numberwithin{equation}{section}

\begin{document}
\title[]{$L_p$-estimates of the conormal derivative problem for parabolic equations with time measurable coefficients and $A_p$-weights}

\author[H. Dong]{Hongjie Dong}
\address[H. Dong]{Division of Applied Mathematics, Brown University, 182 George Street, Providence, RI 02912, USA.}
\email{Hongjie\_Dong@brown.edu}
\thanks{H. Dong was partially supported by the NSF under agreement DMS-2350129.}
\author[P. Jung]{Pilgyu Jung}
\address[P. Jung]{Department of mathematics, Yonsei University, 50 Yonsei-Ro, Seodaemun-Gu, Seoul, 03722, Republic of Korea}
\email{pilgyujung@skku.edu}
\thanks{P. Jung was supported by the National Research Foundation of Korea (NRF) under agreement NRF-2022R1A2C1003322 and RS-2022-NR069609.}
\author[D. Kim]{Doyoon Kim}
\address[D. Kim]{Department of Mathematics, Korea University, 145 Anam-ro, Seongbuk-gu, Seoul, 02841, Republic of Korea}
\email{doyoon\_kim@korea.ac.kr}
\thanks{D. Kim was supported by the National Research Foundation of Korea (NRF) grant funded by the Korea government (MSIT) (RS-2025-16065192).}

\begin{abstract}
This paper investigates weighted mixed-norm estimates for divergence-type parabolic equations on Reifenberg-flat domains with the conormal derivative boundary condition. The leading coefficients are assumed to be merely measurable in the time variable and to have small mean oscillations in the spatial variables. In deriving the boundary estimates, we overcome a regularity issue by employing half-time derivative estimates.
\end{abstract}

\maketitle

\section{introduction}

Over the past few decades, extensive literature has been devoted to $L_p$-theory for elliptic and parabolic equations with irregular leading coefficients and domains. For elliptic equations, numerous results address interior and boundary estimates for the Dirichlet problem. In particular, in \cite{MR1911202,MR1191890,MR1088476,MR1405255}, when the leading coefficients were assumed to belong to the class of vanishing mean oscillation (VMO), the authors employed singular integral techniques, specifically perturbation arguments based on representation formulas of solutions and the Coifman-Rochberg-Weiss commutator theorem. These results were subsequently extended in \cite{MR2110431,MR2069724} to Lipschitz and Reifenberg-flat domains.
Unlike earlier works, these authors did not rely on integral representations of solutions; 
instead, they used a level set argument due to Caffarelli and Peral \cite{MR1486629} and the ``crawling of ink spots'' lemma, originally due to Safonov and Krylov \cite{MR563790}.

For parabolic equations, the development proceeded in parallel with that of elliptic equations. We refer to \cite{MR1239929} for the singular integral approach.
In \cite{MR2304157}, parabolic equations in $\bR^d$ with leading coefficients belonging to VMO in the spatial variables and merely measurable in time were investigated.
The method there makes use of pointwise bounds for sharp functions of derivatives of solutions along with the Hardy-Littlewood maximal function theorem and the Fefferman-Stein theorem for sharp functions.
Inspired by this approach, 
\cite{MR2771670, MR2835999} obtained interior and boundary $L_p$-estimates for elliptic and parabolic equations/systems in $\bR^d$, on a half-space $\bR^d_+$, and on Lipschitz or Reifenberg-flat domains when the homogeneous Dirichlet boundary condition is imposed. 
The coefficients in \cite{MR2771670, MR2835999}  are allowed to be merely measurable in one variable (one spatial variable or the time variable) and have small mean oscillation estimates in the remaining variables.
We refer to such coefficients, in short, as partially small BMO coefficients.

Despite these advances, the $L_p$-theory for parabolic equations remains largely open in the case of Neumann boundary conditions (i.e., conormal derivative boundary conditions) when the coefficients are irregular, especially when they are merely measurable in time, and the domain has an irregular boundary, such as a Reifenberg-flat domain.
In particular, when temporal irregularity of the coefficients is combined with geometric irregularity of the boundary, the problem becomes genuinely non-trivial: one is forced to treat $u_t$ as a distribution rather than as a function belonging to any $L_p$ space.

The novelty of this paper lies in the introduction of fractional time derivatives into the analysis, allowing us to control terms that would otherwise be unmanageable when the coefficients are irregular in time and the boundary is rough.
We study second-order divergence form parabolic equations of the type
\begin{equation}
    \label{eq_intro}
    u_t-D_i(a_{ij}D_j u)+\lambda u= D_t^{1/2} h+D_ig_i+f \quad \text{in $\bR\times \Omega$}
\end{equation}
with the conormal derivative boundary condition on $\bR\times \partial\Omega$ (cf. \eqref{eq0523_04}), where $\lambda \ge0$ and $\Omega$ is a Reifenberg-flat domain.
We aim to derive $L_p$-estimates for parabolic equations with time-irregular leading coefficients (as in \cite{MR2304157}), and moreover to establish weighted mixed-norm estimates as in \cite{MR3812104, MR3947859}.
We emphasize that the equation \eqref{eq_intro} differs from conventional parabolic equations in divergence form, as it contains the so-called half-time derivative of $h$ on the right-hand side.
Indeed, this term is essential in dealing with parabolic equations defined on Reifenberg-flat domains with coefficients $a_{ij}$ that lack regularity in time.
Even in the case $h\equiv 0$, there are, to the best of our knowledge, no existing $L_p$-results for solutions to the equation in this setting and our method relies essentially on half-time derivative estimates of $u$.
The results for \eqref{eq_intro} are then used to derive corresponding results for conventional parabolic equations (without $D_t^{1/2}h$ term) on $(0,T) \times \Omega$, under the conormal derivative boundary condition with the same irregular coefficients and boundary.

To outline the key idea, we consider the following model equation:
\begin{equation}
    \label{model eq}    
u_t-D_i(a_{ij}(t)D_ju)+\lambda u=D_ig_i+f \quad \text{in $\bR\times \Omega$}
\end{equation} 
with the conormal derivative boundary condition on $\bR\times \partial \Omega$ (cf. \eqref{def0814_1}).
Here we set $h=0$ to clarify the argument. Prior to introducing our approach, we review related literature and methods that have been employed in this context.
The corresponding elliptic equations were considered in \cite{MR2107043} for coefficients of small mean oscillations in all variables (small BMO) on Reifenberg-flat domains.
Later, \cite{MR3013054} extended the theory to second- and higher-order elliptic systems with partially small BMO coefficients by avoiding the compactness argument in \cite{MR2107043}, which does not apply to equations/systems with such coefficients.
The key observation (\cite[Lemma 4.2]{MR3013054}) is that, after modifying the right-hand side slightly, the solution $u$ still satisfies the equation with the conormal derivative boundary condition on a subdomain with flat boundary.

For the parabolic case, applying the method in \cite{MR3013054} directly to \eqref{model eq} produces an additional term of the form $u_t I_{\Omega^*}$ on the right-hand side, where $\Omega^*$ denotes the portion of the Reifenberg-flat domain outside a half-space, and $\Omega^*$ has sufficiently small volume. However, solutions to parabolic equations in divergence form generally do not satisfy $u_t \in L_p$, which creates a regularity difficulty. If the leading coefficients are independent of $t$, this issue can be resolved, since $u_t$ is itself a solution whenever $u$ is a homogeneous solution of \eqref{model eq} near the boundary; see \cite[Corollary 4.7]{MR2139880} and \cite[Proposition 5.1]{MR4387198}.

In contrast, when the leading coefficients $a_{ij}$ depend on $t$ without further regularity assumptions, this approach is no longer applicable. This is precisely why no results exist for the $L_p$-theory of the conormal derivative problem with time-irregular leading coefficients. To overcome this obstacle, we employ half-time derivative estimates of $u$, denoted $D_t^{1/2}u$ (see \eqref{eq0513_03} for the definition). Specifically, we rewrite
\[
u_t I_{\Omega^*}=-D_t^{1/2}(H(D_t^{1/2}u) I_{\Omega^*}),
\]
where $H$ denotes the Hilbert transform in $t$.
Consequently, near a subdomain with a flat boundary, equation \eqref{model eq} reduces to
\[
u_t-D_i(a_{ij}(t)D_ju)+\lambda u\approx D_t^{1/2}(H(D_t^{1/2}u)I_{\Omega^*})+D_ig_i+f,
\]
so that even when $h=0$, a term involving the half-time derivative appears on the right-hand side.
We then derive estimates for solutions that include $D_t^{1/2}u$.
In these estimates, the quantity $(D_t^{1/2}u) I_{\Omega^*}$ arises on the right-hand side; to handle this term, we absorb it into the left-hand side by exploiting the smallness of the volume of $\Omega^*$.

The remainder of the paper is organized as follows.
In section \ref{sec main result}, we introduce some notation and function spaces, and then state the main results.
Section \ref{sec auxiliary} is devoted to $L_p$-solvability and local $L_p$-estimates for certain domains.
In Section \ref{sec estimates on a Reifenberg}, we establish boundary estimates on Reifenberg-flat domains using the key idea described above.
Finally, in the last section, we prove the main results by applying the level set argument for the case $p=q$, and then extend the conclusions to $p\neq q$ via an extrapolation theorem.

\section{Function spaces and main results}							\label{sec main result}

\subsection{Basic notation}
Let $d \geq 1$ be a positive integer. Let $\Omega$ be a domain in $\bR^d$, where $\bR^d$ denotes the $d$-dimensional Euclidean space. A point in $\bR^d$ is written as $x = (x_1,x_2,\ldots,x_d) = (x_1,x')$, where $x' \in \bR^{d-1}$.
We set
\[
\bR^d_+= \{x\in \bR^d: x_1>0\} \quad \text{and} \quad \bR^{d+1}_+= \bR\times \bR^{d}_+.
\]
We also set $\bR^d_- = \{x \in \bR^d: x_1 < 0 \}$.
Let 
\[
B_r(x)=\{y\in \bR^d:|x-y|<r\}, \quad B_r=B_r(0),
\]
\[
B_r^+(x)=B_r(x)\cap \bR^d_{+}, \quad B_r^+=B_r^+(0),  \quad \Gamma_r=\overline{B_r}\cap \{x_1=0\}
\]
\[
\Omega_r(x)=B_r(x)\cap \Omega, \quad \Omega_r=\Omega_r(0).
\]
For $X=(t,x)\in \bR^{d+1}$, 
we set
\[
Q_{r,s}(X)=(t-r^2,t+r^2)\times B_{s}(x), \quad Q_{r,s}^+(X)=(t-r^2,t+r^2)\times B^+_{s}(x),
\]
\[
 Q_{r}(X)=Q_{r,r}(X), \quad Q_{r}^+(X)=Q_{r,r}^+(X), \quad Q_{r}=Q_r(0), \quad Q_r^+=Q_r^+(0),
\]
\[
 C_{r,s}(X)=(t-r^2,t+r^2)\times \Omega_{s}(x), \quad C_r(X)=C_{r,r}(X).
\]
Let $E \subset \bR^{d+1}$.
We write $(u)_E$ to denote
\begin{equation*}
(u)_E = \dashint_{E} u(t,x) \, dx \, dt = \dashint_{E} u(t,x) \, dX = \frac{1}{|E|}\int_{E} u(t,x) \, dX,
\end{equation*}
where $|E|$ is the $(d+1)$-dimensional Lebesgue measure of $E$.
As above, we often use $dX$ to denote $dx\,dt$ or $dt \, dx$.
In a similar manner, we define the same averaging notation for integrals with respect to $x \in \bR^d$ or $(t,x') \in \bR \times \bR^{d-1}$.
\subsection{Function spaces}
Let $\cQ = \bR \times \Omega$.
 $C^\infty(\cQ)$ (resp.\ $C^\infty(\overline{\cQ}))$ denotes the set of infinitely differentiable functions on $\cQ$ (resp.\ on its closure \(\overline{\cQ}\)). Let $C_0^\infty(\cQ)$ (resp.\ $C_0^\infty(\overline{\cQ}))$ be the set of all functions in \(C^\infty(\cQ)\) (resp.\ in \(C^\infty(\overline{\cQ})\)) and compactly supported in \(\cQ\) (resp.\ in its closure \(\overline{\cQ}\)).

We define the Muckenhoupt $A_p$-weights. 
Let $p\in (1,\infty)$.
The Muckenhoupt class $A_p(\bR^d)$ consists of all nonnegative functions $\omega(x)$ on $\bR^d$ satisfying
\[
[\omega]_{A_p} := \sup_{x\in \bR^d,r>0} \left( \dashint_{B_r(x)} \omega \, dx \right) \left( \dashint_{B_r(x)} \omega^{-\frac{1}{p-1}} \, dx \right)^{p-1} < \infty.
\]
In the same manner, we define $A_p$-weights on $\bR$. We note that the notion of $A_p$ weights on $\mathbb{R}^{d+1}$ is formulated using parabolic cylinders $Q_r(t,x)$ rather than Euclidean balls.

We are now ready to introduce function spaces. 
Let $p,q\in (1,\infty)$, $\Omega$ be a domain in $\bR^d$, $-\infty\le S<T\le \infty$,  $\omega_1\in A_p(\bR^d)$ and $\omega_2\in A_q(\bR)$.  We set
\[
\omega(t,x):=\omega_1(x)\omega_2(t), \quad t\in \bR, \quad x\in \bR^d.
\]
We denote by $L_{p,q,\omega}((S,T)\times \Omega)$ the set of all measurable functions $u$ on $(S,T)\times \Omega$ having a finite norm
\[
\|u\|_{L_{p,q,\omega}((S,T)\times \Omega)}
=\left(\int_{S}^T\left(\int_{\Omega} |u|^p \omega_1(x)\,dx\right)^{q/p}\omega_2(t)\,dt\right)^{1/q}.
\]
We define $L_{p,\omega_1}(\Omega)$ and $L_{q,\omega_2}((S,T))$ by the same manner.

Next, we define the half-time derivative. 
For $\varphi \in C_0^\infty(\cQ)$ (or for sufficiently smooth $\varphi$ such that the integral below is finite), we define the half-time derivative of $\varphi$ as
\begin{equation*}
D_t^{1/2} \varphi{{(t,x)}} =  \frac{1}{\sqrt{8\pi}} \int_\bR \frac{\varphi(t+\ell,x)-\varphi(t,x)}{|\ell|^{3/2}} \, d\ell.
\end{equation*}
For a general $u\in L_{p,q,\omega}(\cQ)$, we define the half-time derivative in the weak sense. Namely, if there exists a locally integrable function $v$ on $\cQ$ such that 
\begin{equation}
							\label{eq0513_03}
\int_\cQ  u \, D_t^{1/2} \varphi \, dX = \int_\cQ v \, \varphi \, dX
\end{equation}
for all $\varphi \in C_0^\infty(\cQ)$, then $v$ is called the half-time derivative of $u$, and we write
\[
v = D_t^{1/2}u.
\]
We define
\begin{equation*}
H_{p,q,\omega}^{1/2,1}(\cQ)= \{u:u,Du,D_t^{1/2}u\in L_{p,q,\omega}(\cQ)\},
\end{equation*}
equipped with the norm 
\[
\|u\|_{H_{p,q,\omega}^{1/2,1}(\cQ)}=\|D_t^{1/2}u\|_{L_{p,q,\omega}(\cQ)}+\|Du\|_{L_{p,q,\omega}(\cQ)}+\|u\|_{L_{p,q,\omega}(\cQ)}.
\]
Here, $Du=(D_1u,D_2u,\ldots,D_d u)=(D_{x_1}u,D_{x_2}u,\ldots,D_{x_d}u)$ indicates the gradient of $u$ in the spatial variable.

We conclude this section by introducing the function spaces associated with parabolic equations that do not involve half-time derivative terms.
Let
\[
\bH_{p,q,\omega}^{-1}((S,T)\times \Omega)=\left\{u:u=D_ig_i+f:g_i,f\in L_{p,q,\omega}((S,T)\times \Omega))\right\},
\]
\begin{multline*}
\|u\|_{\bH_{p,q,\omega}^{-1}((S,T)\times \Omega)}
\\
=\operatorname{inf}\left\{\sum_{i=1}^d \|g_i\|_{L_{p,q,\omega}((S,T)\times \Omega)}+\|f\|_{L_{p,q,\omega}((S,T)\times \Omega)}:u=D_ig_i+f \right\},
\end{multline*}
\begin{align*}
&\cH_{p,q,\omega}^{1}((S,T)\times \Omega)\\
&=\left\{u:u_t\in\bH_{p,q,\omega}^{-1}((S,T)\times \Omega), Du,u\in L_{p,q,\omega}((S,T)\times \Omega)\right\},\\
&\|u\|_{\cH_{p,q,\omega}^1((S,T)\times \Omega)}\\
&=\|u_t\|_{\bH_{p,q,\omega}^{-1}((S,T)\times \Omega)}+\|Du\|_{L_{p,q,\omega}((S,T)\times \Omega)}+\|u\|_{L_{p,q,\omega}((S,T)\times \Omega)}.
\end{align*}
Let $m\in \mathbb{N}$.
For a multi-index $\alpha=(\alpha_1,\dots,\alpha_d)$, we use the notation $D^\alpha u=D_1^{\alpha_1}\dots D_d^{\alpha_d} u=D_{x_1}^{\alpha_1}\dots D_{x_d}^{\alpha_d} u$. Let
\[
W_{p,q,\omega}^{1,m}\left((S,T)\times \Omega\right)=\{u: u, u_t, D^\alpha u\in L_{p,q,\omega}\left((S,T)\times \Omega\right)\,\,  \text{for} \,\, |\alpha|\le m\}.
\]
Similarly, $W_{p,q,\omega}^{0,m}\left((S,T) \times \Omega\right)$ is defined as above, but without the condition $u_t \in L_{p,q,\omega}\left((S,T) \times \Omega\right)$.

Throughout this paper,  when \( p = q \), we omit \( q \) in the notation of the function space. Moreover, if the weight is equal to 1, we also omit the notation for the weight.

\subsection{Main results}

In this section, we begin by introducing the Hilbert transform to define solutions. Let $\Omega\subset \bR^d$ be a domain and set $\cQ=\bR\times \Omega$. Let $p,q\in (1,\infty)$, $\omega_1\in A_p(\bR^d)$, $\omega_2\in A_q(\bR)$, and denote $\omega(t,x)=\omega_1(x)\omega_2(t)$.  
For $u \in L_{p,q,\omega}(\cQ)$, let $H(u)$ represent the Hilbert transform of $u$ with respect to the time variable defined as follows.
If $u \in C_0^\infty(\cQ)$, one can write
\[
H(u)(t,x) = \lim_{\varepsilon \to 0} H^{(\varepsilon)} (u)(t,x) = \lim_{\varepsilon \to 0} \frac{1}{\pi} \int_{|s| \geq \varepsilon} \frac{u(t-s,x)}{s} \, ds.
\]
 
The boundedness of the Hilbert transform on $L_{p,q,\omega}(\cQ)$ follows immediately in the case $p=q$: one first integrates in $t$ 
and then applies the weighted norm estimate for the Hilbert transform 
from \cite[Theorem 9]{MR312139}. When $p\neq q$, we use the extrapolation theorem in \cite{MR3812104}, which implies there exists a constant $N$ depending only on $d,p,q,[\omega_1]_{A_p},[\omega_2]_{A_q}$ such that
\[
\|H(u)\|_{L_{p,q,\omega}(\cQ)}\le N\|u\|_{L_{p,q,\omega}(\cQ)} \quad \text{for all $u\in L_{p,q,\omega}(\cQ)$.}
\]

We introduce the equations together with the definitions of their corresponding solutions.
For $1\le i,j\le d$, let  $a_{ij}$, $b_{i}$, $c$ be bounded measurable functions on $\cQ$.
Let $-\infty\le S<T\le \infty$. Define $p'$ and $q'$ be the conjugate exponents of $p$ and $q$, respectively and \[\omega'(t,x)=\omega_1'(x)\omega_2'(t)=\omega_1^{\frac{1}{1-p}}(x)\omega_2^{\frac{1}{1-q}}(t),\] so that the dual space of $L_{p,q,\omega}((S,T)\times \Omega)$ is $L_{p',q',\omega'}((S,T)\times \Omega)$.
\begin{definition}
							\label{def0814_2} 
For $h, g_i, f\in L_{p,q,\omega}(\cQ)$,
we say that $u\in H_{p,q,\omega}^{1/2,1}(\cQ)$ is a solution to the equation 
\begin{equation}
\label{def_eq}
u_t-D_i(a_{ij}D_j u)+b_{i}D_iu+c u =D_t^{1/2}h+D_ig_i+f \quad \text{in} \,\, \cQ    
\end{equation}
with the conormal derivative boundary condition on $\bR\times \partial \Omega$
if  for any $\varphi\in H_{p',q',\omega'}^{1/2,1}({\cQ})$, we have 
\begin{multline}
\label{eq0523_04}
\int_\cQ\left(- H(D_t^{1/2}u) \, D_t^{1/2}\varphi +a_{ij}D_juD_i\varphi+b_{i}D_iu \, \varphi+ cu\varphi\right)\,dX
\\
=\int_\cQ \left(h \, D_t^{1/2}\varphi
-g_iD_i\varphi+ f\varphi\right)\,dX.
\end{multline}
\end{definition}
\begin{remark}
In \eqref{eq0523_04}, it is sufficient to take test functions from \( C_0^\infty(\overline{\mathcal{Q}}) \)  instead of \( H_{p',q',\omega'}^{1/2,1}(\mathcal{Q}) \) provided that \( \Omega \) satisfies a suitable flatness condition near the boundary (e.g., Reifenberg-flat domains;  see \cite{MR3186805}). Such geometric assumptions ensure that \( C_0^\infty(\overline{\mathcal{Q}}) \) is dense in \( H_{p',q',\omega'}^{1/2,1}(\mathcal{Q}) \); see, for instance, \cite{MR1774162} and the references therein.
\end{remark}
Since our main results, specifically Theorem \ref{conormal -infty to T} and Corollary \ref{conormal bounded interval} presented below, include the usual parabolic equations (without half-derivative terms), we provide the definition of their solutions as well.
\begin{definition}
							\label{def0814_1}
For $g_i, f \in L_{p,q,\omega}((S,T)\times \Omega)$, we say that $u\in \cH_{p,q,\omega}^{1}((S,T)\times \Omega)$ is a solution to the equation 
\begin{equation}
							\label{eq0819_01}
u_t-D_i(a_{ij}D_j u)+b_{i}D_iu+c u =D_ig_i+f \quad \text{in} \,\, (S,T)\times \Omega
\end{equation}
with the conormal derivative boundary condition on $(S,T)\times \partial \Omega$
if, for any $\varphi \in W^{1,1}_{p',q',\omega'}( (S,T)\times \Omega )$ vanishing at $t=S$ and $t=T$, we have
\begin{multline*}							
\int_{(S,T)\times \Omega}\left(- u \varphi_t+a_{ij}D_juD_i\varphi+b_{i}D_iu \varphi+ cu\varphi \right)\,dX
\\
=\int_{(S,T)\times \Omega}\left(-g_iD_i\varphi
+ f\varphi\right)\,dX.
\end{multline*}
\end{definition}

\begin{remark}
							\label{rem0902_2}
If $u\in H_{p,q,\omega}^{1/2,1}(\cQ)$, then, using an appropriate approximation and the Fourier transform, we have
\[
\int_\cQ H(D_t^{1/2}u) \, D_t^{1/2}\varphi\,dX= \int_\cQ u \varphi_t\,dX
\]
for $\varphi \in W_{p',q',\omega'}^{1,1}(\cQ)$. See, for instance, \cite[Lemmas 3.4, 3.6(4), 3.7]{MR4920684} for the case $p=q$ and $\omega\equiv 1$.
Hence, Definition \ref{def0814_2}, when $h \equiv 0$, corresponds to the definition of weak solutions to the usual parabolic equations with the conormal derivative boundary condition with $S=-\infty$ and $T=\infty$.
\end{remark}

\begin{remark}
							\label{rem1213_1}
In this paper, we sometimes impose the conormal derivative boundary condition on a {\em subset} of the lateral boundary $\bR \times \partial \Omega$.
For instance, if $\cC$ is a subset of $\partial\Omega$, we impose the conormal derivative boundary condition on 
$\bR \times \cC$.
To do this, we use Definition \ref{def0814_2} with test functions $\varphi$ that, if the restriction or trace is well-defined, are zero on $\bR \times (\partial\Omega \setminus \cC)$.
Thus, if we additionally impose the Dirichlet boundary condition on $\bR \times (\partial\Omega \setminus \cC)$, the boundary condition becomes a mixed boundary condition.
If $\cC = \emptyset$, that is, if the test functions are chosen to be zero on $\bR \times \partial\Omega$, then the  equality \eqref{eq0523_04} means that the equation in \eqref{def_eq} holds with no boundary conditions being specified (i.e., the interior case).
More precise descriptions will be provided whenever the conormal derivative boundary condition is imposed on a subset of the lateral boundary. We also consider equations as in \eqref{eq0819_01} with the conormal derivative boundary condition imposed on a subset of the lateral boundary.
\end{remark}

Throughout the paper, we impose the strong ellipticity condition and boundedness on the coefficients $a_{ij}$ with $\delta>0$. Specifically, we assume that
\begin{equation}
    \label{eq0521_02}
\delta|\xi|^2 \le a_{ij}(t,x)\xi_i\xi_j, \quad |a_{ij}(t,x)|\le \delta^{-1}, 
\end{equation}
for all $\xi\in \bR^d$, and $(t,x)\in \bR^{d+1}$.

Now, we state our regularity assumption on the coefficients $a_{ij}$ and the boundary of the domain $\Omega$.
\begin{assumption}[$\gamma_0$]                                        \label{assum rei coeffi}
Let $\gamma_0\in (0,1/600)$. 
There is a constant $R_0\in (0,1]$ such that the following hold.
\begin{itemize}
    \item[(i)]  For any $(t,x) \in \bR \times \Omega$ and $r\in
(0,R_0\wedge \operatorname{dist}(x,\partial\Omega)/2]$,  we have
\begin{equation}
                            \label{coeffi bmo}
\frac{1}{|Q_r|}\int_{Q_r(X)} \Big| a_{ij}(s,y) -
\frac{1}{|B_r|}\int_{B_r(x)} a_{ij}(s,z) \, dz \Big| \, dy\,ds\le
\gamma_{0},
\end{equation}
where $X=(t,x)$.
\item[(ii)] For any $(t,x) \in \bR \times \partial\Omega$ and $r\in
(0,R_0]$, there is a spatial coordinate system depending on $x$ and $r$
such that in this new coordinate system, we have \eqref{coeffi bmo} 
and
\begin{equation*}
 \{(y_1,y'):x^1<y^1\}\cap B_r(x)
 \subset\Omega_r(x)
 \subset \{(y_1,y'):x^1-2\gamma_0 r<y^1\}\cap B_r(x).
\end{equation*}
\end{itemize}
\end{assumption}

\begin{remark}
							\label{rem0902_1}
From \eqref{coeffi bmo}, it follows that, for any $a, b \in \bR$ with $b - a \geq 2r^2$, we have
\[
\frac{1}{b-a}\int_a^b \dashint_{B_r(x)} \left|a_{ij}(s,y) - \dashint_{B_r(x)} a_{ij}(s,z) \, dz \right| \, dy \, ds \leq 2 \gamma_{0}.
\]
See \cite[Remark 2.4]{MR4920684}.
\end{remark}
We are ready to introduce our main results.
Recall that $\cQ = \bR \times \Omega$.

\begin{theorem}
    \label{conormal entire time}
    Let $p,q\in(1,\infty)$, $\omega_1=\omega_1(x)\in A_p(\bR^d)$, $\omega_2=\omega_2(t)\in A_q(\bR)$, $[\omega_1]_{A_p}+[\omega_2]_{A_q}\le K$ for some constant $K\ge 1$, $\lambda \geq 0$, $h,g_i,f\in L_{p,q,\omega}(\cQ)$ with $f \equiv 0$ if $\lambda = 0$. Denote $\omega(t,x)=\omega_1(x)\omega_2(t)$.
Then, there exist positive constants $\gamma_{0}=\gamma_{0}(d,\delta,p,q,K)$ and $N=N(d,\delta,p,q,K)$ such that, under Assumption \ref{assum rei coeffi} ($\gamma_{0}$), for any $u\in H_{p,q,\omega}^{1/2,1}(\cQ)$ satisfying 
        \begin{equation}
        \label{main eq}
    u_t-D_i(a_{ij}D_j u)+\lambda u=D_t^{1/2}h+D_ig_i+f \quad \text{in} \,\, \cQ,
    \end{equation}
    with the conormal derivative boundary condition on $\bR\times\partial\Omega$,
we have 
\begin{multline*}
\|D_t^{1/2}u\|_{L_{p,q,\omega}(\cQ)} + \|Du\|_{L_{p,q,\omega}(\cQ)}+\sqrt{\lambda}\|u\|_{L_{p,q,\omega}(\cQ)}
\\
\leq N \left( \|h\|_{L_{p,q,\omega}(\cQ)} + \|g_i\|_{L_{p,q,\omega}(\cQ)} + \lambda^{-1/2}\|f\|_{L_{p,q,\omega}(\cQ)}\right),
\end{multline*}
provided that $\lambda \ge \lambda_0$, where $\lambda_0=\lambda_0(d,\delta,p,q,K, R_0)\geq0$.
Moreover, for $\lambda  > \lambda_0$, there exists a unique solution $u\in H_{p,q,\omega}^{1/2,1}(\cQ)$ to  \eqref{main eq} with the conormal derivative boundary condition on $\bR\times \partial\Omega$.
\end{theorem}
Using Theorem \ref{conormal entire time}, we establish $L_p$-theory for the conventional parabolic conormal problem with time-irregular coefficients $a_{ij}$ on $\Omega_T=(-\infty,T)\times \Omega$ for $T\in (-\infty,\infty]$.

\begin{theorem}
\label{conormal -infty to T}
    Let $p,q\in(1,\infty)$, $T\in (-\infty,\infty]$, $\omega_1=\omega_1(x)\in A_p(\bR^d)$, $\omega_2=\omega_2(t)\in A_q(\bR)$, $[\omega_1]_{A_p}+[\omega_2]_{A_q}\le K$ for some constant $K\ge 1$, $\lambda\ge0$, $g_i, f \in L_{p,q,\omega}(\Omega_T)$ with $f\equiv 0$ if $\lambda=0$. Denote $\omega(t,x)=\omega_1(x)\omega_2(t)$. Then, there exist positive constants $\gamma_{0}=\gamma_{0}(d,\delta,p,q,K)$ and $N=N(d,\delta,p,q,K)$ such that, under Assumption \ref{assum rei coeffi} ($\gamma_{0}$), for any $u\in \cH_{p,q,\omega}^{1}(\Omega_T)$ satisfying 
        \begin{equation}
        \label{main eq -infty to T}
    u_t-D_i(a_{ij}D_j u)+\lambda u=D_ig_i+f \quad \,\,{in} \,\, \Omega_T,
    \end{equation}
    with the conormal derivative boundary condition on $(-\infty,T)\times \partial\Omega$,
        we have 
\begin{equation}
							\label{eq0906_01}
\|Du\|_{L_{p,q,\omega}(\Omega_T)}+\sqrt{\lambda}\|u\|_{L_{p,q,\omega}(\Omega_T)}\le N \left( \|g_i\|_{L_{p,q,\omega}(\Omega_T)}+\lambda^{-1/2}\|f\|_{L_{p,q,\omega}(\Omega_T)}\right),
\end{equation}
provided that $\lambda \ge \lambda_0$, where $\lambda_0=\lambda_0(d,\delta,p,q,K,R_0)\ge0$.
Moreover, for $\lambda>\lambda_0$, there exists a unique solution $u\in  \cH_{p,q,\omega}^{1}(\Omega_T)$ to  \eqref{main eq -infty to T} with the conormal derivative boundary condition on $(-\infty,T)\times \partial\Omega$.
\end{theorem}

Finally, we state our results for the corresponding Cauchy problem with $T \in (0,\infty)$, taking as the solution space 
$\cH_{p,q,\omega,0}^1((0,T)\times \Omega)$, 
the subspace of $\cH_{p,q,\omega}^{1}((0,T)\times \Omega)$ of functions $u$ satisfying $u(0,x)=0$.
\begin{corollary}
    \label{conormal bounded interval}
    Let $p,q\in(1,\infty)$, $T\in (0,\infty)$, $\omega_1=\omega_1(x)\in A_p(\bR^d)$, $\omega_2=\omega_2(t)\in A_q(\bR)$, $[\omega_1]_{A_p}+[\omega_2]_{A_q}\le K$ for some constant $K\ge 1$, $g_i, f \in L_{p,q,\omega}((0,T)\times \Omega)$. Denote $\omega(t,x)=\omega_1(x)\omega_2(t)$. Then, there exist positive constants $\gamma_{0}=\gamma_{0}(d,\delta,p,q,K)$ and $N=N(d,\delta,p,q,K,R_0,T)$ such that, under Assumption \ref{assum rei coeffi} ($\gamma_{0}$), for any $u\in \cH_{p,q,\omega,0}^{1}((0,T)\times \Omega)$ satisfying 
        \begin{equation}
        \label{main eq bdd interval}
    u_t-D_i(a_{ij}D_j u)=D_ig_i+f \quad \,\,{in} \,\, (0,T)\times \Omega,
    \end{equation}
    with the conormal derivative boundary condition on $(0,T)\times \partial\Omega$,
        we have 
\begin{multline}
							\label{eq0910_9}
\|Du\|_{L_{p,q,\omega}((0,T)\times \Omega)}+\|u\|_{L_{p,q,\omega}((0,T)\times \Omega)}
\\
\le N \left( \|g_i\|_{L_{p,q,\omega}((0,T)\times \Omega)}+\|f\|_{L_{p,q,\omega}((0,T)\times \Omega)}\right).
\end{multline}
Moreover, there exists a unique solution $u\in  \cH_{p,q,\omega,0}^{1}((0,T)\times \Omega)$ to \eqref{main eq bdd interval} with the conormal derivative boundary condition on $(0,T)\times \partial\Omega$.
\end{corollary}
\begin{remark} 
In Theorems \ref{conormal entire time}, \ref{conormal -infty to T}, and Corollary \ref{conormal bounded interval}, 
the weight $\omega_1$ is defined on $\mathbb{R}^d$, rather than on $\Omega$. 
The class $A_p(\Omega)$ consists of all nonnegative functions $\omega(x)$ on $\Omega$ such that
\[
[\omega]_{A_p(\Omega)} := \sup_{x\in \Omega,r>0} \left( \dashint_{\Omega_r(x)} \omega \, dx \right) \left( \dashint_{\Omega_r(x)} \omega^{-\frac{1}{p-1}} \, dx \right)^{p-1} < \infty.
\]
We also define 
\[
[\omega]_{A_{p,\Omega}} :=
\sup_{x\in \bR^d,r>0} \left( \frac{1}{|B_r|}\int_{\Omega_r(x)} \omega \, dx \right) \left( \frac{1}{|B_r|}\int_{\Omega_r(x)} \omega^{-\frac{1}{p-1}} \, dx \right)^{p-1}.
\]
Note that 
\begin{equation}
    \label{eq0911_01}
    [\omega]_{A_{p,\Omega}}\le 2^{dp} [\omega]_{A_p(\Omega)}.
\end{equation}
Indeed, for any $x\in \bR^d$ and $r>0$ with $\Omega_r(x)\neq \emptyset$, we choose $y \in \Omega_r(x)$. Then $\Omega_r(x)\subset \Omega_{2r}(y)$, and hence 
\begin{align*}
 \left( \frac{1}{|B_r|}\int_{\Omega_r(x)} \omega \, dx \right)& \left( \frac{1}{|B_r|}\int_{\Omega_r(x)} \omega^{-\frac{1}{p-1}} \, dx \right)^{p-1}
\\    
&\le \left( \frac{1}{|B_r|}\int_{\Omega_{2r}(y)} \omega \, dx \right) \left( \frac{1}{|B_r|}\int_{\Omega_{2r}(y)} \omega^{-\frac{1}{p-1}} \, dx \right)^{p-1} 
\\&= 2^{dp}\left( \frac{1}{|B_{2r}|}\int_{\Omega_{2r}(y)} \omega \, dx \right) \left( \frac{1}{|B_{2r}|}\int_{\Omega_{2r}(y)} \omega^{-\frac{1}{p-1}} \, dx \right)^{p-1}
\\&\le 2^{dp}[w]_{A_p(\Omega)}.
\end{align*}
It is well known that if $\omega$ is a nonnegative function on $\Omega$ such that 
$[\omega^{1+\varepsilon}]_{A_{p,\Omega}}<\infty$ for some $\varepsilon>0$, 
then $\omega$ can be extended to a weight $\bar \omega \in A_p(\mathbb{R}^d)$; see, for instance, \cite{MR807149, MR4340793}. 

If $\Omega$ satisfies the doubling property, i.e., there exists a constant $K_1>0$ such that for any $x \in \Omega$ and $r>0$,
\[
0 < |\Omega_{2r}(x)| \leq K_1 |\Omega_r(x)| < \infty,
\]
then $\Omega$, endowed with the Euclidean distance and Lebesgue measure, becomes a homogeneous space. 
In this case, for any $\omega \in A_p(\Omega)$, by the reverse H\"older inequality, there exists $\varepsilon>0$ such that $\omega^{1+\varepsilon}\in A_p(\Omega)$. 
By \eqref{eq0911_01}, we then obtain $[\omega^{1+\varepsilon}]_{A_{p,\Omega}}<\infty$, so that $\omega$ admits an extension. 
Therefore, if we further assume that $\Omega$ satisfies the doubling property, the conclusions of Theorems \ref{conormal entire time}, \ref{conormal -infty to T}, 
and Corollary \ref{conormal bounded interval} remain valid with $A_p(\mathbb{R}^d)$ 
replaced by $A_p(\Omega)$.
For the same reason, the statement of Corollary \ref{conormal bounded interval} also holds for 
$\omega_2 \in A_q((0,T))$ instead of $A_q(\bR)$.
\end{remark}
\begin{remark}
The main results also extend to systems, since our arguments remain valid in the corresponding setting. 
For the sake of convenience, however, we restrict our presentation to the case of equations. In the above theorems and corollary, one may also consider equations with lower-order terms, as in \eqref{def_eq} and \eqref{eq0819_01}, under appropriate conditions on $b_i$ and $c$, such as boundedness.
For simplicity, in this paper, we restrict to the case $b_i = c = 0$.
\end{remark}

\section{Auxiliary results}\label{sec auxiliary}
In this section, we first prove Theorem \ref{conormal entire time} for $p=2$ with $w\equiv 1$, which is equivalent to establishing the $L_2$-solvability of the conormal derivative boundary value problem.
Note that, in the theorem below, no regularity assumptions are imposed on either the coefficients or the boundary of the domain.
\begin{theorem}
							\label{thm1126_1}
Let $\Omega$ be a domain in $\bR^d$, $\cQ=\bR\times \Omega$, and let $a_{ij}$ satisfy the conditions in $\eqref{eq0521_02}$.
There exists $N=N(d,\delta)$ such that, for any $\lambda\ge 0$,
\begin{equation*}
\|D_t^{1/2}u\|_2 + \|Du\|_2 + \sqrt{\lambda}\|u\|_2 \leq N \left( \|h\|_2 + \|g_i\|_2 + \lambda^{-1/2}\|f\|_2\right),
\end{equation*}
where $\|\cdot\|_2 = \|\cdot\|_{L_2(\cQ)}$, provided that $u \in H_2^{1/2,1}(\cQ)$, $h, g_i, f \in L_2(\cQ)$ with $f \equiv 0$ if $\lambda = 0$, and 
\begin{equation}
							\label{eq1119_01}
u_t - D_i(a_{ij}D_ju) + \lambda u = D_t^{1/2}h+D_ig_i+f \quad \text{in} \,\, \cQ
\end{equation}
with the conormal derivative boundary condition on $\bR \times \partial \Omega$.
Furthermore, for any $\lambda>0$, $h, g_i, f \in L_2(\cQ)$, there exists a unique solution $u\in H_2^{1/2,1}(\cQ)$ to the equation \eqref{eq1119_01} with the conormal derivative boundary condition on $\bR \times \partial\Omega$.
\end{theorem}

\begin{proof}
The proof is identical to that of \cite[Theorem 4.1]{MR4920684}, which is based on the Lax-Milgram theorem.
In particular, to check 
\[
\int_\cQ H(D_t^{1/2}u) \, D_t^{1/2}u \, dX = 0
\]
for $u \in H_2^{1/2,1}(\cQ)$, one may use the Fourier transform. See, for instance, \cite[Lemma 3.4, Lemma 3.6(4)]{MR4920684}.
\end{proof}

Next, we establish the unique solvability of equations in $L_p$ spaces, $1<p<\infty$, when the domain is $\bR \times \bR^d$ and the coefficients $a_{ij}$ satisfy the following assumption (as well as \eqref{eq0521_02}):

\begin{assumption}
							\label{assum0103_1}
$a_{ij}$ satisfy either (i) or (ii) of the following:
\begin{enumerate}
\item[(i)] $a_{11} = a_{11}(t)$, $a_{ij} = a_{ij}(t,x_1)$ for $(i,j) \neq (1,1)$.
\item[(ii)] $a_{11} = a_{11}(x_1)$, $a_{ij} = a_{ij}(t,x_1)$ for $(i,j) \neq (1,1)$.
\end{enumerate}
\end{assumption}

Note that the coefficients in the assumption above depend only on $t$, $x_1$, or $(t,x_1)$ and possess no regularity.

\begin{proposition}
							\label{prop0105_1}
Let $p \in (1,\infty)$, $\lambda \geq 0$, $\cQ = \bR \times \bR^d$, and $h, g_i, f \in L_p(\cQ)$ with $f \equiv 0$ if $\lambda = 0$.
Assume that $a_{ij}$ satisfy Assumption \ref{assum0103_1}.
Then, there exists $N = N(d,\delta,p)$ such that, for any $u \in H_p^{1/2,1}(\cQ)$ satisfying
\begin{equation}
							\label{eq0103_01}
u_t - D_i(a_{ij}D_ju) + \lambda u = D_t^{1/2}h + D_i g_i + f \quad \text{in} \,\, \cQ,
\end{equation}
we have
\begin{equation}
							\label{eq0105_01}
\|D_t^{1/2}u\|_p + \|Du\|_p + \sqrt{\lambda}\|u\|_p \leq N \left(\|h\|_p + \|g_i\|_p + \lambda^{-1/2}\|f\|_2\right),
\end{equation}
where $\|\cdot\|_p = \|\cdot\|_{L_p(\cQ)}$.
Moreover, for $\lambda > 0$, there exists a unique solution $u \in H_p^{1/2,1}(\cQ)$ satisfying \eqref{eq0103_01}.
\end{proposition}

\begin{proof}
Since the case $\lambda = 0$ follows by letting $\lambda \searrow 0$, we assume $\lambda > 0$.

We first consider the case where $a_{ij}=\delta_{ij}$, i.e., the Laplace operator.
In this case, by \cite[Theorems 2.7 or 2.9]{MR4920684}, the desired estimate and solvability in the proposition follow.
Indeed, the results in \cite[Theorems 2.7 or 2.9]{MR4920684} hold for $\lambda \geq 
\lambda_0(d,\delta,p)$, where $\lambda_0$ may not be zero.
For $0 < \lambda < \lambda_0(d,\delta,p)$, we use a scaling
\[
t\to (\lambda_0/\lambda)t,\quad x\to \sqrt{\lambda_0/\lambda}\, x
\]
to obtain the estimate \eqref{eq0105_01}.
For the solvability, see the proof of \cite[Theorem 2.7]{MR4920684}.

For general $a_{ij}$ as in Assumption \ref{assum0103_1}, due to the unique solvability of the equation \eqref{eq0103_01} with $a_{ij}=\delta_{ij}$ just proved above together with the method of continuity, it suffices to prove the estimate \eqref{eq0105_01}.
Let $u \in H_p^{1/2,1}(\cQ)$ satisfy \eqref{eq0103_01}.
Using the results proved above, we first find $v \in H_p^{1/2,1}(\cQ)$ satisfying
\[
v_t - \Delta v + \lambda v = D_t^{1/2}h \quad \text{in} \,\, \cQ
\]
with the estimate
\begin{equation}
							\label{eq0105_03}
\|D_t^{1/2}v\|_p + \|D v\|_p + \sqrt{\lambda}\|v\|_p \leq N \|h\|_p,
\end{equation}
where $N = N(d,\delta,p)$.
Set $w = u - v \in H_p^{1/2,1}(\cQ)$, which satisfies
\begin{equation}
							\label{eq0105_02}
w_t - D_i(a_{ij}D_j w) + \lambda w = D_i( g_i - D_i v + a_{ij}D_jv) + f \quad \text{in} \,\, \cQ.
\end{equation}
Since $w, Dw \in L_p(\cQ)$ and, from \eqref{eq0105_02}, $w_t$ can be written as $w_t = D_i G_i + F$, where $G_i, F \in L_p(\cQ)$ (see Remark \ref{rem0902_2}), it follows by the definition of $\cH_p^1(\cQ)$ that $w \in \cH_p^1(\cQ)$.
On the other hand, for the equation \eqref{eq0105_02}, whose coefficients satisfy Assumption \ref{assum0103_1}, by \cite[Theorem 6.1 (iii)]{MR2764911}, we have the unique solvability in $\cH_p^1(\cQ)$.
Hence, $w \in \cH_p^1(\cQ) \cap H_p^{1/2,1}(\cQ)$ is the unique solution to \eqref{eq0105_02}
with the estimate
\begin{equation}
							\label{eq0105_04}
\|Dw\|_p + \sqrt{\lambda}\|w\|_p \leq N \left( \|g_i\|_p + \|Dv\|_p +  \lambda^{-1/2}\|f\|_p \right),
\end{equation}
where $N = N(d,\delta,p)$.

Combining the estimates \eqref{eq0105_03} and \eqref{eq0105_04}, along with $u=w+v$, we obtain
\begin{equation}
							\label{eq0105_05}
\|Du\|_p + \sqrt{\lambda}\|u\|_p \leq N \left(\|h\|_p + \|g_i\|_p + \lambda^{-1/2}\|f\|_p\right),
\end{equation}
where $N = N(d,\delta,p)$.
To estimate $D_t^{1/2}u$, we rewrite the equation \eqref{eq0103_01} as
\[
u_t - \Delta u + \lambda u = D_t^{1/2}h + D_i(g_i - D_i u + a_{ij}D_ju) + f \quad \text{in} \,\, \cQ
\]
and combine the estimate proved for $a_{ij}=\delta_{ij}$ and the estimate \eqref{eq0105_05}.
The proposition is thus proved.
\end{proof}

We now present a version of Proposition \ref{prop0105_1} for the case where the spatial domain is either the half space $\bR^d_+$ or a partially bounded domain:
\[
\Pi_R = \{x=(x_1,x') \in \bR^d: 0<x_1<R, \, x' \in \bR^{d-1}\} = (0,R) \times \bR^{d-1} \subset \bR^d.
\]
Due to the presence of the boundary of the spatial domain, we impose a boundary condition on the lateral boundary $\bR \times \partial\Omega$, where $\Omega = \bR^d_+$ or $\Omega = \Pi_R$.
The boundary condition we consider is either the zero Dirichlet boundary condition or the conormal derivative boundary condition on $\bR \times \partial\Omega$.
If $\Omega = \Pi_R$, we also consider the mixed boundary condition, i.e., the conormal derivative boundary condition on
 $\bR \times \{(0,x'): x' \in \bR^{d-1}\}$ and the zero Dirichlet boundary condition on $\bR \times \{(R,x'): x' \in \bR^{d-1}\}$.
In this case, as mentioned in Remark \ref{rem1213_1}, the notion of solutions is defined as in Definition \ref{def0814_2} with test functions $\varphi \in H_q^{1/2,1}(\bR \times \Pi_R)$ satisfying $
\varphi(t,R,x') = 0$ for $(t,x') \in \bR \times \bR^{d-1}$.

\begin{proposition}
							\label{prop1211_1}
Let $p \in (1,\infty)$, $R > 0$, $\lambda \geq 0$, $\Omega = \bR^d_+$ or $\Omega = \Pi_R$, $\cQ = \bR \times \Omega$, and $h, g_i, f \in L_p(\cQ)$ with $f \equiv 0$ if $\lambda = 0$.
Assume that $a_{ij}$ satisfy Assumption \ref{assum0103_1}.

\begin{enumerate}
\item[(i)] For $\Omega = \bR^d_+$, there exists $N = N(d,\delta,p)$ such that, for any $u \in H_p^{1/2,1}(\cQ)$ satisfying \eqref{eq0103_01} with either the zero Dirichlet boundary condition or the conormal derivative boundary condition on $\bR \times \partial\Omega$, we have \eqref{eq0105_01}.
Moreover, for $\lambda > 0$, there exists a unique solution $u \in H_p^{1/2,1}(\cQ)$ to \eqref{eq0103_01} with either the zero Dirichlet boundary condition or the conormal derivative boundary condition on $\bR \times \partial\Omega$.

\item[(ii)] For $\Omega = \Pi_R$, we have the same estimate and solvability result as above with the conormal derivative boundary condition on $\bR \times \{(0,x'): x' \in \bR^{d-1}\}$ and the zero Dirichlet boundary condition on $\bR \times \{(R,x'): x' \in \bR^{d-1}\}$.
\end{enumerate}
\end{proposition}

\begin{remark}
When $\Omega = \Pi_R$, one can also impose, instead of the mixed boundary condition described above, either the zero Dirichlet boundary condition or the conormal derivative boundary condition on $\bR \times \partial\Omega$.
In particular, for the conormal derivative boundary condition on $\bR \times \partial\Omega$, we rely on the $L_2$-solvability result (Theorem \ref{thm1126_1}); see the proof below.
For $L_2$-solvability under the zero Dirichlet boundary condition, see \cite[Theorem 4.1]{MR4920684}.
\end{remark}

\begin{proof}[Proof of Proposition \ref{prop1211_1}]
When $\Omega = \bR^d_+$, the results follow from Proposition \ref{prop0105_1} and the extension argument (for instance, see \cite[Theorem 6.1]{MR4345837} and the references therein).
Indeed, to employ the extension argument, we need the following extension of $a_{ij}$:
\[
\bar{a}_{ij} = \bar{a}_{ij} \quad \text{for} \quad (i,j) = (1,1) \quad \text{or} \quad i, j \geq 2,
\]
\[
\bar{a}_{1j} = \left\{
\begin{aligned}
a_{1j}(t,x_1), &\quad x_1 > 0,
\\
-a_{1j}(t,x_1), &\quad x_1 < 0,
\end{aligned}
\right.
\quad
\bar{a}_{j1} = \left\{
\begin{aligned}
a_{j1}(t,x_1), &\quad x_1 > 0,
\\
-a_{j1}(t,x_1), &\quad x_1 < 0,
\end{aligned}
\right.
\]
for $j=2,\ldots,d$.
Note that $\bar{a}_{ij}$ again satisfy Assumption \ref{assum0103_1} as well as the ellipticity condition \eqref{eq0521_02}.


For $\Omega = \Pi_R$, again by the extension argument (more precisely, the even extension of $u$ with respect to $x_1$), it is enough to consider equations with coefficients as in Assumption \ref{assum0103_1} in $\Xi_R$ with the zero Dirichlet boundary condition on $\bR \times \partial\Xi_R$, where
\[
\Xi_R = \{x = (x_1,x') \in \bR^d: -R < x_1 < R, x' \in \bR^{d-1}\}.
\]
Then, we obtain the desired results by following the proof of \cite[Proposition 6.2]{MR4345837} and using Proposition \ref{prop0105_1}.
More precisely, we employ the periodic extensions of the equation with respect to $x_1$ to the whole space $\bR \times \bR^d$ and multiply it by appropriate cut-off functions to derive the desired results.
In particular, the proof of the existence of solutions begins with the $L_2$-solvability result established in \cite[Theorem 4.1]{MR4920684} for the zero Dirichlet boundary condition, where no regularity assumptions are needed for either $a_{ij}$ or the boundary of the spatial domain.
\end{proof}

From Proposition \ref{prop1211_1}, we derive the following local estimate.

\begin{lemma}
							\label{lem0106_1}
Let $p \in (1,\infty)$, $R \in (0,\infty)$, $\cQ = \bR \times \Pi_R$, and $a_{ij}$ satisfy Assumption \ref{assum0103_1}.
Then, there exits $N=N(d,\delta,p)$ such that, for any $\lambda \geq 0$,
\begin{align*}
&\left(2^{-1}R^{-2}\int_{-R^2}^{R^2} \int_{\Pi_R} |D_t^{1/2}u|^p + |Du|^p + \lambda^{p/2}|u|^p \, dx \, dt\right)^{1/p}\\
&\leq N \sum_{j=0}^\infty 2^{-j/4} \left(2^{-j-1}R^{-2} \int_{-2^j R^2}^{2^j R^2} \int_{\Pi_R} |h|^p  + |g_i|^p + \lambda^{-p/2}|f|^p \, dx \, dt\right)^{1/p},
\end{align*}
provided that $u \in H_p^{1/2,1}(\cQ)$, $h, g_i, f \in L_p(\cQ)$ with $f \equiv 0$ if $\lambda = 0$, and
\[
u_t - D_i(a_{ij}D_j u) + \lambda u = D_t^{1/2} h + D_i g_i + f \quad \text{in} \,\, \cQ
\]
with the conormal derivative boundary condition on $\bR \times \{(0,x'): x' \in \bR^{d-1}\}$ and the zero Dirichlet boundary condition on $\bR \times \{(R,x'): x' \in \bR^{d-1}\}$.
\end{lemma}

\begin{proof}
Thanks to scaling, we consider only the case $R = 1$.
Let $\lambda_k = 2^{-k}$ and 
$\eta_k$, $k=0,1,2,\ldots$, be a $C^\infty(\bR)$ function such that $0 \leq \eta_k \leq 1$,
\[
\eta_k(t) = \left\{
\begin{aligned}
1 \quad & \text{for} \,\, t \in (-2^k,2^k),
\\
0 \quad & \text{for} \,\, t \in \bR \setminus (-2^{k+1},2^{k+1}),
\end{aligned}
\right.
\]
and $|\partial_t \eta_k| \leq N 2^{-k}$, where $N$ is a constant independent of $k$.
Then, $u \eta_k$ satisfies
\begin{multline}
							\label{eq1211_01}
(u \eta_k)_t - D_i\left( a_{ij} D_j (u\eta_k) \right) + (\lambda + \lambda_k) u \eta_k 
\\
= D_t^{1/2}(h \eta_k) + D_i(g \eta_k) + f \eta_k + h_k + u (\eta_k)_t + \lambda_k u \eta_k
\end{multline}
in $\cQ$ with the conormal derivative condition on $\bR \times \{(0,x'): x' \in \bR^{d-1}\}$ and the zero Dirichlet boundary condition on $\bR \times \{(1,x'): x' \in \bR^{d-1}\}$, where
\begin{multline}
							\label{eq1211_02}
h_k(t,x) = \eta_k D_t^{1/2}h - D_t^{1/2}(\eta_k h)
\\
= - \frac{1}{\sqrt{8 \pi}} \int_\bR h(t+\ell,x)\frac{\eta_k(t+\ell)- \eta_k(t)}{|\ell|^{3/2}} \, d\ell.
\end{multline}
Note that the last two terms in \eqref{eq1211_01} can be rewritten as
\[
u(\eta_k)_t + \lambda_k u \eta_k = u (\eta_k)_t \eta_{k+1} + \lambda_k u \eta_k \eta_{k+1} \quad \text{in}\,\, \cQ.
\]
Then, by Proposition \ref{prop1211_1} applied to \eqref{eq1211_01} along this observation, we have
\begin{multline}
							\label{eq0106_01}
\|D_t^{1/2}(u\eta_k)\|_p + \|D(u\eta_k)\|_p + \sqrt{\lambda + \lambda_k}\|u \eta_k\|_p 
\\
\leq N \|h\eta_k\|_p + N\|g\eta_k\|_p + \frac{N}{\sqrt{\lambda+\lambda_k}}\|F_k\|_p,
\end{multline}
where $\|\cdot\|_p = \|\cdot\|_{L_p(\cQ)}$, $N = N(d,\delta,p)$, and
\[
F_k = f \eta_k + h_k + u(\eta_k)_t \eta_{k+1} + \lambda_k u \eta_k \eta_{k+1}.
\]
Note that
\begin{equation}
							\label{eq0106_02}
\frac{1}{\sqrt{\lambda + \lambda_k}}\|F_k\|_p \leq \frac{1}{\sqrt{\lambda}}\|f \eta_k\|_p + \frac{1}{\sqrt{\lambda + \lambda_k}}\|h_k\|_p + N\frac{2^{-k}+\lambda_k}{\sqrt{\lambda+\lambda_k}}\|u \eta_{k+1}\|_p,
\end{equation}
where, by the choice of $\lambda_k$,
\[
\frac{1}{\sqrt{\lambda + \lambda_k}} \leq 2^{k/2}, 
\quad \frac{2^{-k}+\lambda_k}{\sqrt{\lambda+\lambda_k}} \leq 2 \cdot 2^{-k/2}.
\]
From this, \eqref{eq0106_02}, and \eqref{eq0106_01}, it follows that
\begin{align*}
&\|D_t^{1/2}(u\eta_k)\|_p + \|D(u\eta_k)\|_p + \sqrt{\lambda}\|u \eta_k\|_p\\
&\leq \|D_t^{1/2}(u\eta_k)\|_p + \|D(u\eta_k)\|_p + \sqrt{\lambda + \lambda_k}\|u \eta_k\|_p \\
&\leq N \|h \eta_k\|_p + N\|g\eta_k\|_p + \frac{N}{\sqrt{\lambda}}\|f \eta_k\|_p + N 2^{k/2}\|h_k\|_p + N 2^{-k/2}\|u \eta_{k+1}\|_p.
\end{align*}
As $h_k$ in \eqref{eq1211_02}, we have
\[
\eta_k D_t^{1/2}u = D_t^{1/2}(u\eta_k) - u_k,
\]
where
\[
u_k = \frac{1}{\sqrt{8\pi}} \int_\bR u(t+\ell,x)\frac{\eta_k(t+\ell)-\eta_k(t)}{|\ell|^{3/2}} \, d\ell.
\]
From this and the inequalities above, we obtain that
\begin{multline}
							\label{eq1211_03}
\|\eta_k D_t^{1/2}u\|_p + \|\eta_k Du\|_p + \sqrt{\lambda}\|\eta_k u\|_p
\leq \|u_k\|_p + N \|h \eta_k\|_p + N \|g\eta_k\|_p
\\
+ \frac{N}{\sqrt{\lambda}}\|f\eta_k\|_p + N 2^{k/2}\|h_k\|_p + N 2^{-k/2}\|u \eta_{k+1}\|_p.
\end{multline}
From \cite[Lemma 3.3]{MR4920684}, we see that
\begin{align*}
&\|u_k\|_p \leq N 2^{-k/2}\sum_{j=1}^\infty 2^{-j(1/2+1/p)}\|u\|_{L_p\left( (-2^{k+j}, 2^{k+j}) \times \Pi_1 \right)},\\
&\|h_k\|_p \leq N 2^{-k/2}\sum_{j=1}^\infty 2^{-j(1/2+1/p)}\|h\|_{L_p\left( (-2^{k+j}, 2^{k+j}) \times \Pi_1 \right)},
\end{align*}
where $N = N(p)$.
Combining the above two inequalities with \eqref{eq1211_03} and recalling the properties of $\eta_k$, we have
\begin{multline}
							\label{eq1213_01}
\|D_t^{1/2}u\|_{L_p\left((-2^k,2^k) \times \Pi_1\right)} + \|Du\|_{L_p\left((-2^k,2^k) \times \Pi_1\right)} + \sqrt{\lambda}\|u\|_{L_p\left((-2^k,2^k) \times \Pi_1\right)}
\\
\leq
N 2^{-k/2}\sum_{j=1}^\infty 2^{-j(1/2+1/p)}\|u\|_{L_p\left( (-2^{k+j}, 2^{k+j}) \times \Pi_1 \right)}
\\
+ N \sum_{j=1}^\infty 2^{-j(1/2+1/p)}\||h| + |g_i| + \lambda^{-1/2}|f| \|_{L_p\left( (-2^{k+j}, 2^{k+j}) \times \Pi_1 \right)},
\end{multline}
where by the Poincar\'{e} inequality with the fact that $u(t,1,x')=0$ for $(t,x') \in \bR \times \bR^{d-1}$,
\begin{equation}
							\label{eq1213_03}
\|u\|_{L_p\left( (-2^{k+j}, 2^{k+j}) \times \Pi_1 \right)} \leq \|D_1u\|_{L_p\left( (-2^{k+j}, 2^{k+j}) \times \Pi_1 \right)}.
\end{equation}
Denote
\begin{align*}
&A_k = \left(2^{-k-1} \int_{-2^k}^{2^k} \int_{\Pi_1} |D_t^{1/2}u|^p + |Du|^p + \lambda^{p/2}|u|^p \, dx \, dt\right)^{1/p},\\
&B_k = \left(2^{-k-1} \int_{-2^k}^{2^k} \int_{\Pi_1} |h|^p + |g_i|^p + \lambda^{-p/2}|f|^p \, dx \, dt\right)^{1/p}.
\end{align*}
Then, the inequality \eqref{eq1213_01} along with \eqref{eq1213_03} implies that
\begin{equation}
							\label{eq1213_02}
A_k \leq N \sum_{j=k+1}^\infty 2^{-j/2} A_j + N 2^{k/2} \sum_{j=k+1}^\infty 2^{-j/2} B_j, \quad k=0,1,2,\ldots,
\end{equation}
where $N = N(d,\delta,p)$.
Choose a positive integer $k_0$ such that
\[
N \sum_{k=k_0}^\infty 2^{-k/2} \leq 1/2.
\]
We then multiply \eqref{eq1213_02} by $2^{-k/2}$ and sum over $k=k_0, k_0+1, \ldots$, to obtain
\[
\sum_{k=k_0}^\infty 2^{-k/2}A_k \leq N \sum_{k=k_0}^\infty 2^{-k/2}\sum_{j=k+1}^\infty 2^{-j/2} A_j + N \sum_{k=k_0}^\infty \sum_{j=k+1}^\infty 2^{-j/2}B_j
\]
\[
\leq N \sum_{j=k_0+1}^\infty 2^{-j/2}A_j \sum_{k=k_0}^{j-1} 2^{-k/2} + N \sum_{j=k_0+1}^\infty 2^{-j/2} B_j \sum_{k=k_0}^{j-1} 1
\]
\[
\leq \frac{1}{2} \sum_{j=k_0+1}^\infty 2^{-j/2}A_j + N \sum_{j=k_0+1}^\infty 2^{-j/4}B_j.
\]
It follows that
\[
\sum_{k=k_0}^\infty 2^{-k/2}A_k \leq N \sum_{j=0}^\infty 2^{-j/4}B_j.
\]
Finally, for the values $k=k_0-1, k_0-2, \ldots,0$, we invoke \eqref{eq1213_02} and use induction to have
\[
A_0 \leq N \sum_{j=0}^\infty 2^{-j/4} B_j,
\]
which is the desired inequality for $R=1$.
\end{proof}


In the remaining section, we consider the operator
\[
\overline{\cP}_\lambda u:=u_t-D_i(a_{ij}(t)D_j u)+\lambda u,
\]
where the coefficients $a_{ij}(t)$ only satisfy \eqref{eq0521_02} without any regularity assumptions.

We now use the notion of solutions given in Definition \ref{def0814_1}.
As in Remark \ref{rem1213_1}, we also consider equations with the conormal derivative boundary condition on a subset of the lateral boundary.
In particular, if the domain is $Q_R^+$ and the conormal derivative boundary condition is imposed on $(-R^2,R^2) \times \Gamma_R$, then the test functions in Definition \ref{def0814_1} are those that vanish on $(-R^2,R^2) \times (\partial B_R \cap \bR^d_+)$.
Note that the function space considered here is $\cH_p^1$, 
whereas Proposition \ref{prop1211_1} is concerned with $H_p^{1/2,1}$.

\begin{lemma}
Let $p\in (1,\infty)$ and $\lambda \ge 0$. Assume $u\in \cH_p^1(\bR^{d+1}_+)$ satisfies 
\[
\overline{\cP}_\lambda u =D_ig_i+f \quad \text{in} \,\, \bR^{d+1}_+
\]
with the conormal derivative boundary condition on $\bR\times\partial \bR^d_+$, where $g_i, f \in L_p(\bR^{d+1}_+)$ with $f \equiv 0$ if $\lambda = 0$.
Then we have
       \begin{equation}
           \label{half cono lp esti}       \|Du\|_{L_p(\bR^{d+1}_+)}+\sqrt{\lambda}\|u\|_{L_p(\bR^{d+1}_+)}\le N\|g_i\|_{L_p(\bR^{d+1}_+)}+\frac{N}{\sqrt{\lambda}}\|f\|_{L_p(\bR^{d+1}_+)},
              \end{equation}
where $N=N(d,\delta,p)$.
\end{lemma}

\begin{proof}
As in the proof of Proposition \ref{prop0105_1}, we assume $\lambda > 0$.
Then, the estimate \eqref{half cono lp esti} follows from \cite[Proposition 7.2]{MR2764911}, which in fact states that the estimate holds for $\lambda \geq \lambda_0(d,\delta,p)$, where the latter may not be zero.
For $0 < \lambda < \lambda_0(d,\delta,p)$, we use the scaling argument in the proof of Proposition \ref{prop0105_1}.
The lemma is proved.
\end{proof}

Using the above $L_p$-estimates on the half space, we obtain the following local estimates on the half space. 
\begin{lemma}
    \label{lem half cono local}
    Let $R\in(0,\infty)$ and $p\in (1,\infty)$. Assume that $u\in \cH_p^1(Q_{2R}^+)$ satisfies 
\begin{equation*}
\overline{\cP}_0 u=0 \quad \text{in} \,\, Q_{2R}^+
\end{equation*}
with the conormal derivative boundary condition on $(-4R^2,4R^2)\times \Gamma_{2R}$. Then, we have 
    \[
    \|Du\|_{L_p(Q_R^+)}\le N\|u\|_{L_p(Q_{2R}^+)},
    \]
    where $N=N(d,\delta,p,R)$.
\end{lemma}
\begin{proof}
Since the estimate \eqref{half cono lp esti} is available, we follow the proof of \cite[Corollary 3.2]{MR2764911} with $p$ in place of $2$.
Indeed, by multiplying $u$ by an appropriate cut-off function $\zeta$ with compact support in $Q_{2R}$, we have that $\zeta u$ satisfies a parabolic equation in $\bR \times \bR^d_+$ with the conormal derivative boundary condition on $\bR \times \partial \bR^d_+$.
Then, we use the estimate \eqref{half cono lp esti} for $\zeta u$ with an iteration argument.
We omit the details.
\end{proof}

We state a parabolic version of Sobolev-Poincar\'e embeddings on a half space. More general results are in \cite[Lemma 3.8]{MR4387198}. The main ingredients of the proof are embedding results in \cite{MR4387945}.

\begin{lemma}
    \label{lem half cono embed}
    Let $R\in(0,\infty)$ and $1<p< p_{1}<\infty$. Suppose that
    \[
    \frac{1}{p}< \frac{1}{d+2}+\frac{1}{p_{1}},
    \]
    and
    $u\in W_{p}^{0,1}({Q_{2R}^+})$ satisfies 
    \[
     u_t=D_ig_i \quad \text{in} \,\, Q_{2R}^+
    \]
    in the distribution sense, where $g=(g_1,\dots g_d)\in (L_p(Q_{2R}^+))^d$. Then, we have 
    \[
    \|u-(u)_{Q_R^+}\|_{L_{p_{1}}(Q_R^+)}\le N(\|Du\|_{L_p(Q_{2R}^+)}+\|g_i\|_{L_p(Q_{2R}^+)}),
    \]
    where $N=N(d,\delta,p,p_{1},R)$.
\end{lemma}

Now, we are prepared to achieve arbitrary degrees of integrability for solutions to homogeneous equations with the conormal derivative boundary condition on a half space. 

\begin{lemma}
\label{lem half cono local lambda}
Let $p \in (1,\infty)$.
Assume that $u\in\cH_p^1(Q_{4R})$ satisfies 
\begin{equation}
							\label{eq0819_04}
\overline{\cP}_\lambda u=0 \quad \text{in} \,\, Q_{4R}^+
\end{equation}
with the conormal derivative boundary condition on $\left(-(4R)^2,(4R)^2\right) \times \Gamma_{4R}$.
Then, for any $p_{1} \in (p,\infty)$, we have 
\[
(|Du|^{p_{1}})_{Q_{R}^+}^{1/p_{1}}+\sqrt{\lambda}(|u|^{p_{1}})_{Q_{R}^+}^{1/{p_{1}}}\le N(|Du|^p)_{Q_{4R}^+}^{1/p}+N\sqrt{\lambda}(|u|^p)_{Q_{4R}^+}^{1/p},
\]
where $N=N(d,\delta,p,p_{1})$.
\end{lemma}
\begin{proof}
    Utilizing an idea by S. Agmon (see, for example, \cite[Lemma 6.3.8]{MR2435520}), we assume that $\lambda=0$. By scaling, we also assume that $R=1$. 
For $p_{1} \in (p,\infty)$, suppose that
\[
\frac{1}{p} < \frac{1}{d+2} + \frac{1}{p_{1}}, \quad \text{i.e.}, \quad \left\{
\begin{aligned}
p_{1} < \frac{p(d+2)}{d+2-p} \quad &\text{if} \quad p \in (1,d+2),
\\
p_{1} \in (p,\infty) \quad &\text{if} \quad p \in [d+2,\infty).
\end{aligned}
\right.
\]

We first prove that $u \in\cH_{p_{1}}^1(Q_r^+)$ for any $r \in (0,4)$.
Since $u\in\cH_p^1(Q_4^+)$, by the parabolic embedding (see, for instance, \cite[Theorem 5.2]{MR4387945}), we have $u \in L_{p_{1}}(Q_4^+)$.
Denote
\[
\mathfrak{Q}:= \left(-4^2,4^2\right)  \times \bR^d_+.
\]
Multiplying a cut-off function $\zeta$ with compact support in $Q_4$, we see that $\zeta u \in \cH_p^1(\mathfrak{Q})$ satisfies
\[
(\zeta u)_t - D_i\left(a_{ij}(t) D_j(\zeta u)\right) = D_i g_i + f
\]
in $\mathfrak{Q}$ with the conormal derivative boundary condition on $(-4^2,4^2)  \times \partial\bR^d_+$ and the zero initial condition $(\zeta u)(-4^2,x) = 0$, where
\[
g_i = - a_{ij}(t) u D_j \zeta, \quad f = u \zeta_t - D_i \zeta a_{ij}(t) D_j u.
\]
Note that $f \in L_p(\mathfrak{Q})$ and $g_i \in L_{p_{1}}(\mathfrak{Q})$; the latter is because $u \in L_{p_{1}}(Q_4^+)$.
Then, by \cite[Theorem 2.15]{MR4387945} combined with \cite[Theorem 2.5]{MR2764911}, there exists a unique $w \in W_{p_{1}}^{0,1}(\mathfrak{Q})$ satisfying
\begin{equation}
							\label{eq0828_01}
w_t - D_i(a_{ij}(t)D_jw) = D_ig_i + f
\end{equation}
in $\mathfrak{Q}$ the with conormal derivative boundary condition on $(-4^2,4^2)  \times \partial\bR^d_+$ and the zero initial condition $w(0,x) = 0$.
Then, since $g_i, f \in L_p(\mathfrak{Q})$, we have $w\in \cH_{p}^1(\mathfrak{Q})$. Using the unique solvability of \eqref{eq0828_01} in $\cH_p^1(\mathfrak{Q})$ with the same initial and boundary conditions, one can show that $w = \zeta u$.
This shows that
\[
\zeta u, D(\zeta u) \in L_{p_{1}}(\mathfrak{Q}).
\]
If we assume $\zeta \equiv 1$ on $Q_r$, then $u, Du \in L_{p_{1}}(Q_r^+)$.
Moreover, the equation \eqref{eq0819_04} $u_t = D_i\left(a_{ij}(t) D_ju\right)$
in $Q_r^+$ in the distribution sense means that $u \in \cH_{p_{1}}^1(Q_r^+)$ because $D_ju \in L_{p_{1}}(Q_r^+)$.

Upon noticing that $u-(u)_{Q_2^+} \in \cH_{p_{1}}^1(Q_2^+)$ also satisfies the equation \eqref{eq0819_04} with the same boundary condition, we 
apply Lemma \ref{lem half cono local} to $u-(u)_{Q_2^+}$ and Lemma \ref{lem half cono embed} to $u_t = D_i(a_{ij}(t) D_ju)$, where $u$ also belongs to $W_p^{0,1}(Q_4^+)$, to get
\[
\|Du\|_{L_{p_{1}}(Q_1^+)}\le N\|u-(u)_{Q_2^+}\|_{L_{p_{1}}(Q_2^+)}\le N\|Du\|_{L_p(Q_4^+)}.
\]
If $p \in (1,d+2)$ and $p_{1} \geq p(d+2)/(d+2-p)$, as in the standard bootstrap argument, we repeat the above process as many times as necessary to reach the desired $p_{1}$.
The lemma is proved. 
\end{proof}

\begin{remark}
							\label{rem0903_1}
The above lemma holds for $u \in H_2^{1/2,1}(\bR \times B_{4R}^+)$ satisfying
\begin{equation}
							\label{eq0903_01}
\overline{\cP}_\lambda u=0 \quad \text{in} \,\, \bR \times B_{4R}^+
\end{equation}
with the conormal derivative boundary condition on $\bR \times \Gamma_{4R}$.
Indeed, from Remark \ref{rem0902_2} along with the fact that $u$ satisfies \eqref{eq0903_01}, it follows that $u$ also belongs to $\cH_2^1(Q_{4R}^+)$ and satisfies $\overline{\cP}_\lambda u=0$ in $Q_{4R}^+$ with the conormal derivative boundary condition on $\left(-(4R)^2,(4R)^2\right) \times \Gamma_{4R}$ in the sense of Definition \ref{def0814_1} (or as in Lemma \ref{lem half cono local lambda}).
\end{remark}

So far, we have obtained boundary estimates.
Here, we present the corresponding interior estimate.
The proof is based on the same argument.
Thus, we state the result without providing a proof.

\begin{lemma}
     Let $p\in (1,\infty)$. Assume that $u\in \cH_p^1(Q_{4R})$ satisfies 
\begin{equation}
							\label{eq0819_05}
\overline{\cP}_\lambda u=0 \quad \text{in} \,\, Q_{4R}.
\end{equation}
Then, for any $p_{1}\in (p,\infty)$, we have 
\begin{equation}
							\label{eq0114_01}
(|Du|^{p_{1}})_{Q_{R}}^{1/p_{1}}+\sqrt{\lambda}(|u|^{p_{1}})_{Q_{R}}^{1/p_{1}}\le N(|Du|^p)_{Q_{4R}}^{1/p}+N\sqrt{\lambda}(|u|^p)_{Q_{4R}}^{1/p},
\end{equation}
where $N=N(d,\delta,p)$.
\end{lemma}

\begin{remark}
Indeed, it is possible to have $p_{1} = \infty$ in \eqref{eq0114_01}, corresponding to the $L_\infty$-estimate for the solution to the homogeneous equation \eqref{eq0819_05} as all spatial derivatives are also solutions to the homogeneous equation. See the proof of \cite[Lemma 4]{MR2771670} for the details. Consequently, unlike in Lemma \ref{lem half cono local lambda}, the constant $N$ in the above estimate depends only on $d, \delta, p$.
\end{remark}

\section{Estimates on a Reifenberg-flat domain}\label{sec estimates on a Reifenberg}

In this section, we obtain interior and boundary estimates for solutions to equations on a Reifenberg-flat domain.

\subsection{Estimates of \texorpdfstring{$u$}{\it u} and \texorpdfstring{$Du$}{\it Du} on a Reifenberg-flat domain}

Throughout this subsection, we consider the parabolic operator
\[
\cP_\lambda u :=u_t-D_i(a_{ij}D_j u)+\lambda u, 
\]
where the coefficients $a_{ij}$ satisfy the conditions in \eqref{eq0521_02} and Assumption \ref{assum rei coeffi}.
We also assume that the spatial domain $\Omega$ satisfies Assumption \ref{assum rei coeffi}.

\begin{lemma}
							\label{lem0901_1}
Let $\lambda \in (0,\infty)$, $\nu \in (1,\infty)$, $\cQ=\bR\times \Omega$, and $(t_0,x_0) \in \bR \times \partial\Omega$.
Assume that $u\in H_\nu^{1/2,1}(\cQ)$ satisfies
\begin{equation}
							\label{eq0829_01}
\cP_\lambda u=D_t^{1/2} h +D_ig_i+f \quad \text{in $\cQ$}
\end{equation}
with the conormal derivative boundary condition on $\bR\times {\partial\Omega}$, where $h, g_i, f \in L_\nu(\cQ)$.
Also, assume that for $R \in (0,R_0]$, where $R_0$ is from Assumption \ref{assum rei coeffi}, by rotation and translation (see Figure \ref{fig1}),
\begin{equation}
							\label{eq0829_03}
\{(y_1,y'):y_1 > 0\}\cap B_R(x_0) \subset\Omega_R(x_0) \subset \{(y_1,y'):y_1 > -2\gamma_0R\}\cap B_R(x_0).
\end{equation}
In particular, the first coordinate of $x_0$ is negative, so that $B_R^+(x_0) \subset B_R^+$. 
Set
\[
\bar{a}_{ij}(t)=\frac{1}{|B_R(x_0)|}\int_{B_R(x_0)}a_{ij}(t,y)\,dy,
\]
and denote
\[
\overline{\cP}_\lambda u = u_t - D_i\left(\bar{a}_{ij}D_ju\right) + \lambda u.
\]
Then, $u \in H_\nu^{1/2,1}\left(\bR \times B_R^+(x_0)\right)$ satisfies 
    \begin{equation}
        \label{eq of u on half space}
\overline{\cP}_\lambda u =D_t^{1/2}\tilde{h}+D_i\tilde{g}_i+\tilde{f} \quad \text{in}\,\, \bR\times B_R^+(x_0)
    \end{equation}
with the conormal derivative boundary condition on $\bR\times (\{x_1=0\}\cap \overline{B_R(x_0)})$.
In \eqref{eq of u on half space}, $\tilde{h}$, $\tilde{g}_i$, and $\tilde{f}$ are defined as follows:
For $(t,x_1,x')\in \bR\times B_R^+(x_0)$,
\[
\begin{aligned}
\tilde{h}(t,x_1,x') &= h(t,x_1,x')+h(t,-x_1,x')I_{\Omega^*_R(x_0)}(-x_1,x')
\\
&+H(D_t^{1/2}u)(t,-x_1,x')I_{\Omega^*_R(x_0)}(-x_1,x'),
\\
\tilde{g}_i(t,x_1,x')
&= g_i(t,x_1,x')+(-1)^{\delta_{1i}}g_i(t,-x_1,x')I_{\Omega^*_R(x_0)}(-x_1,x')
\\
&+({a}_{ij}-\bar{a}_{ij})D_j u + \sum_{j=1}^d(-1)^{\delta_{1i}} (a_{ij}D_ju)(t,-x_1,x') I_{\Omega^*_R(x_0)}(-x_1,x'),
\\
\tilde{f}(t,x_1,x')&=f(t,x_1,x')+f(t,-x_1,x')I_{\Omega^*_R(x_0)}(-x_1,x')
\\
&-\lambda  u(t,-x_1,x')I_{\Omega^*_R(x_0)}(-x_1,x'),
\end{aligned}
\]
where $\delta_{1 i}=1$ if $i=1$ and $\delta_{1 i}=0$ if $i\neq 1$, and
\begin{equation}
							\label{eq0901_01}
\Omega_R^*(x_0) = \bR^d_- \cap \Omega_R(x_0).
\end{equation}
\end{lemma}

\begin{figure}
       \centering
   \begin{tikzpicture}
      \clip (1,0) rectangle (12,5.5);
      \draw (6,1) circle (3.5);
       \draw (2,2) -- (10,2);
        \draw  (2,1)-- (2.5,0.7)--(3,1.2)--(3.5,1.3)--(4,0.9)--(4.5,0.7)--(5,1.2)--(5.5,0.9)--(6,1)--(6.5,0.6)--(7,1.1)--(7.5,1)--(8,1.1)
        (8,1.1)--(8.5,1)
        (8.5,1)--(9,0.9)--(9.5,0.7)--(10,1);
       \draw (2,0) -- (10,0);
       \draw  (9.5,2) arc[start angle=0, end angle=180, radius=3.5cm];
       \draw  (7.75,2) arc[start angle=0, end angle=180, radius=1.75cm];
       \node at (6,1) {$\bullet$};
       \node at (6.2,1.2) {$x_0$};
       \node at (6,2) {${\bullet}$}; 
       \node at  (6.2,2.3){$\hat{x}_0=0$};
         \node at (11,2) {$x_1=0$};
       \node at (11,0.3) {$x_1=-2\gamma_0 R$};
       \node at (10.4,1) {$\partial \Omega$};
       \node at (6,4) {\small{$B_{R}^{+}(x_0)$}};
       \node at (8.5,1.5) {\small{${\Omega_R^-(x_0)}$}};
       \node at (6,5) {\small{$B_R^+$}};
       \node at (6,3) {\small{$B_{R/2}^+$}};
       \end{tikzpicture}     
       \caption{}
       \label{fig1}
\end{figure}

\begin{proof}
Set
\[
\cQ^+_R = \bR \times B_R^+(x_0), \quad
\cO = \bR^d_+\cap \partial B_R(x_0), \quad \cD = \bR \times \cO.
\]
According to Definition \ref{def0814_2} and Remark \ref{rem1213_1}, we need to verify the integral formulation in \eqref{eq0523_04} with $\cQ$ replaced by $\cQ_R^+$ for test functions $\varphi \in H_{\nu'}^{1/2,1}(\cQ^+_R)$, $1/\nu + 1/\nu' = 1$, such that $\varphi = 0$ on $\cD$.
By the choice of such $\varphi$, we observe that the zero extension $\bar{\varphi}$ of $\varphi$ to $\bR \times \left( \bR^d_+ \setminus B_R^+(x_0) \right)$ belongs to $H_{\nu'}^{1/2,1}(\bR \times \bR^d_+)$.
Denoting again by $\bar{\varphi}$, the even extension of $\bar{\varphi}$ with respect to $x_1$, we see that $\bar{\varphi} \in H_{\nu'}^{1/2,1}(\bR \times \bR^d)$.
In particular, $\bar{\varphi} \in H_{\nu'}^{1/2,1}(\cQ)$.

On the other hand, the equation \eqref{eq0829_01} can be written as
\[
\overline{\cP}_\lambda u = D_t^{1/2}h + D_i\left(g_i + (a_{ij} - \bar{a}_{ij})D_ju\right) + f
\]
in $\cQ$ with the conormal derivative boundary condition on $\bR \times \partial\Omega$.
Then, applying the test function $\bar{\varphi}$ to the above equation, we get
\begin{multline}
							\label{eq0829_02}
\int_{\cQ} \left( - H (D_t^{1/2}u) D_t^{1/2}\bar{\varphi} + \bar{a}_{ij}D_ju D_i \bar{\varphi} + \lambda u \bar{\varphi} \right) \, dX
\\
= \int_{\cQ} \left(h D_t^{1/2}\bar{\varphi} - \left(g_i +  (a_{ij} - \bar{a}_{ij})D_ju \right) D_i \bar{\varphi} + f \bar{\varphi} \right)\, dX.
\end{multline}
Considering the support of $\bar{\varphi}$, the left-hand side (LHS) of \eqref{eq0829_02} can be written as
\[
\text{LHS of \eqref{eq0829_02}} = \int_{\cQ^+_R} \left(- H (D_t^{1/2}u) D_t^{1/2}\varphi + \bar{a}_{ij}D_ju D_i \varphi + \lambda u \varphi\right) \, dX
\]
\[
+ \int_{\bR \times \Omega^*_R(x_0)} \left( - H (D_t^{1/2}u) D_t^{1/2}\bar{\varphi} + \bar{a}_{ij}D_ju D_i \bar{\varphi} + \lambda u \bar{\varphi} \right) \, dX =: J_1 + J_2,
\]
where
\[
J_2 = \int_{\bR \times \bR^d_-}  \left(- H (D_t^{1/2}u) D_t^{1/2}\bar{\varphi} + \bar{a}_{ij}D_ju D_i \bar{\varphi} + \lambda u \bar{\varphi}\right)I_{\Omega_R^*(x_0)}(x_1,x') \, dX
\]
\[
= \int_{\bR \times \bR^d_+} - H (D_t^{1/2}u)(t,-x_1,x') I_{\Omega_R^*(x_0)}(-x_1,x') D_t^{1/2}\varphi(t,x_1,x')  \, dX
\]
\[
+ \int_{\bR \times \bR^d_+}  (-1)^{\delta_{1i}} \bar{a}_{ij}D_ju(t,-x_1,x') I_{\Omega_R^*(x_0)}(-x_1,x') D_i \varphi(t,x_1,x')  \, dX
\]
\[
+ \int_{\bR \times \bR^d_+} \lambda u(t,-x_1,x') I_{\Omega_R^*(x_0)}(-x_1,x') \varphi(t,x_1,x') \, dX
\]
\[
= \int_{\cQ^+_R} \cdots \, dX + \int_{\cQ^+_R} \cdots \, dX + \int_{\cQ^+_R} \cdots \, dX.
\]
The last equality above is due to the fact that the support of $\varphi$ is contained in $\bR \times\overline{B_R^+(x_0)}$.
Similarly to $J_1$ and $J_2$, the right-hand side (RHS) of \eqref{eq0829_02} can be written as
\[
\text{RHS of \eqref{eq0829_02}} = \int_{\cQ^+_R} \left( h D_t^{1/2}\varphi - (g_i+(a_{ij} - \bar{a}_{ij}) D_ju) D_i \varphi + f \varphi\right) \, dX
\]
\[
+ \int_{\cQ^+_R} h(t,-x_1,x')I_{\Omega_R^*(x_0)}(-x_1,x') D_t^{1/2}\varphi(t,x_1,x') \, dX
\]
\[
- \int_{\cQ^+_R} (-1)^{\delta_{1i}} g_i(t,-x_1,x') I_{\Omega_R^*(x_0)}(-x_1,x') D_i\varphi(t,x_1,x') \, dX
\]
\[
- \int_{\cQ^+_R} (-1)^{\delta_{1i}} \left(a_{ij}(t,-x_1,x') - \bar{a}_{ij}(t)\right) D_ju(t,-x_1,x') I_{\Omega_R^*(x_0)}(-x_1,x') D_i\varphi(t,x_1,x') \, dX
\]
\[
+ \int_{\cQ^+_R} f \varphi \, dX + \int_{\cQ^+_R} f(t,-x_1,x') I_{\Omega_R^*(x_0)}(-x_1,x') \varphi(t,x_1,x') \, dX.
\]
From the above observations on both sides of \eqref{eq0829_02},  we see that \eqref{eq0829_02} can be rewritten as
\begin{align*}
&\int_{\cQ^+_R} \left(- H (D_t^{1/2}u) D_t^{1/2}\varphi + \bar{a}_{ij}D_ju D_i \varphi + \lambda u \varphi\right) \, dX\\
&= \int_{\cQ^+_R} 
\left(\tilde{h} D_t^{1/2}\varphi - \tilde{g}_i D_i \varphi + \tilde{f} \varphi \right) \, dX.
\end{align*}
This implies the assertion of the lemma.
\end{proof}

\begin{proposition}
\label{prop reifen} 
Let $\lambda \in (0,\infty)$, $p \in (1,\infty)$, $\nu \in(p,\infty)$, $\nu'=\nu p/(\nu-p)$, $R\in(0,R_0]$, $\cQ=\bR\times \Omega$, and $p_{1}\in (p,\infty)$. Assume that  $u \in H_\nu^{1/2,1}(\cQ) \cap H_p^{1/2,1}(\cQ)$ satisfies
\begin{equation*}
    \cP_\lambda u=D_t^{1/2} h +D_ig_i+f \quad \text{in $\cQ$}
\end{equation*}
with the conormal derivative boundary condition on $\bR\times {\partial\Omega}$, where $h, g_i, f \in L_\nu(\cQ) \cap L_p(\cQ)$.
Denote
\begin{equation}
							\label{eq0108_03}
\cU = (Du,\sqrt{\lambda}u), \quad F = (h,g_i, \lambda^{-1/2}f).
\end{equation}
Then, under Assumption \ref{assum rei coeffi} ($\gamma$), there exists $N=N(d,\delta,p,\nu, p_{1})$ such that the following hold true. 
\begin{itemize}
\item[(i)] 
For $X_0=(t_0,x_0)\in \bR\times \Omega$ with $\operatorname{dist}(x_0,\partial\Omega)\ge R$, there exist vector-valued functions $\cW$ and $\cV$ in $Q_R(X_0)$ such that $\cU = \cW + \cV$ in $Q_R(X_0)$, where $\cW$ and $\cV$ satisfy 
\begin{align*}
\left(|\cW|^p\right)_{Q_R(X_0)}^{1/p}
&\le N\gamma_0 ^{1/\nu'}\sum_{j=0}^\infty2^{-j/4}\left(|\cU|^\nu\right)^{1/\nu}_{Q_{2^{j/2}R,R}(X_0)}\\
&\quad + N\sum_{j=0}^\infty2^{-j/4}\left(|F|^p\right)^{1/p}_{Q_{2^{j/2}R,R}(X_0)},
\end{align*}
and
\begin{align*}
\left(|\cV|^{p_{1}}\right)_{Q_{R/8}(X_0)}^{1/p_{1}} &\le N\gamma_0^{1/\nu'}\sum_{j=0}^\infty2^{-j/4}\left(|\cU|^\nu\right)^{1/\nu}_{Q_{2^{j/2}R,R}(X_0)} \\
&\quad + N\sum_{j=0}^\infty2^{-j/4}\left(|F|^p\right)^{1/p}_{Q_{2^{j/2}R,R}(X_0)} +N\left(|\cU|^p\right)^{1/p}_{Q_R(X_0)}.
\end{align*}

\item[(ii)]
For $X_0=(t_0,x_0)\in \bR\times \partial\Omega$,
there exist vector-valued functions $\cW$ and $\cV$ in $C_R(X_0)$ such that $\cU = \cW + \cV$ in $C_R(X_0)$, where $\cW$ and $\cV$ satisfy
\begin{equation}
							\label{eq0108_01}
\begin{aligned}
\left(|\cW|^p\right)_{C_R(X_0)}^{1/p}
&\le N\gamma_0^{1/\nu'}\sum_{j=0}^\infty  2^{-j/4}\left(|\cU|^\nu\right)^{1/\nu}_{C_{2^{j/2}R,R}(X_0)}
\\
&+ N\sum_{j=0}^\infty  2^{-j/4}\left(|F|^p\right)^{1/p}_{C_{2^{j/2}R,R}(X_0)}
\\
+ N\gamma_0^{1/\nu'}&\sum_{j=0}^\infty  2^{-j/4}\left(|H(D_t^{1/2}u)|^\nu\right)^{1/\nu}_{C_{2^{j/2}R,R}(X_0)},
\end{aligned}
\end{equation}
and
\begin{equation}
							\label{eq0108_02}
\begin{aligned}
\left(|\cV|^{p_{1}}\right)^{1/p_{1}}_{C_{R/8}(X_0)}\le N\gamma_0^{1/\nu'}\sum_{j=0}^\infty  2^{-j/4}\left(|\cU|^{\nu}\right)^{1/\nu}_{C_{2^{j/2}R,R}(X_0)}
\\
+ N\sum_{j=0}^\infty  2^{-j/4}\left(|F|^p\right)^{1/p}_{C_{2^{j/2}R,R}(X_0)} + N\left(|\cU|^p\right)^{1/p}_{C_R(X_0)}
\\
+ N\gamma_0^{1/\nu'}\sum_{j=0}^\infty  2^{-j/4}\left(|H(D_t^{1/2}u)|^\nu\right)^{1/\nu}_{C_{2^{j/2}R,R}(X_0)}.
\end{aligned}
\end{equation}
\end{itemize}
\end{proposition}

\begin{proof}
Since case (i) is relatively easier, our focus here is on case (ii).
Due to Assumption \ref{assum rei coeffi}, we assume that \eqref{eq0829_03} is satisfied.
For convenience of notation, 
we further assume that, by translation, the last $d-1$ coordinates of $x_0$ are zero and $t_0 = 0$.
In particular, the origin is located at $\hat{x}_0$ (see Figure \ref{fig1}), i.e., $\hat{x}_0 = 0 \in \bR^d$.

By Lemma \ref{lem0901_1}, $u \in H_\nu^{1/2,1}\left(\bR \times B_R^+(x_0)\right) \cap H_p^{1/2,1}\left(\bR \times B_R^+(x_0)\right)$ satisfies \eqref{eq of u on half space} with the boundary condition and the functions
\[
\tilde{h}, \tilde{g}_i, \tilde{f} \in L_\nu\left(\bR \times B_R^+(x_0)\right) \cap L_p\left(\bR \times B_R^+(x_0)\right)
\]
described in the lemma.
We note that $B^+_R(x_0)\subset B_R^+$.
We assign the value zero to $\tilde{h}$, $\tilde{g}$, and $\tilde{f}$ on $\bR\times (\Pi_R\setminus B_R^+(x_0))$.
Since $\bar{a}_{ij}(t)$ satisfy Assumption \ref{assum0103_1}, using Proposition \ref{prop1211_1}, we find $w \in H_\nu^{1/2,1}(\bR \times \Pi_R) \cap H_p^{1/2,1}(\bR \times \Pi_R)$ satisfying
\begin{equation}
							\label{eq0107_01}
\overline{\cP}_\lambda w = D_t^{1/2}\tilde{h} + D_i \tilde{g}_i + \tilde{f} \quad \text{in} \,\, \bR \times \Pi_R
\end{equation}
with the conormal derivative condition on $\bR \times \{(0,x'): x' \in \bR^{d-1}\}$ and the zero Dirichlet boundary condition on $\{(R,x'): x' \in \bR^{d-1}\}$.
Applying Lemma \ref{lem0106_1} to \eqref{eq0107_01} along with fact that $\tilde{h}, \tilde{g}_i, \tilde{f}$ on $\bR \times (\Pi_R \setminus B_R^+)$ are zero, we have
\begin{equation}
      \label{eq0107_03}
\left(|Dw|^p + \lambda^{p/2}|w|^p\right)^{1/p}_{Q_R^+} \le N\sum_{j=0}^\infty  2^{-j/4}\left(|\tilde{F}|^p\right)^{1/p}_{Q_{2^{j/2}R,R}^+ },
\end{equation}
where $\tilde{F}=(\tilde{h},\tilde{g}, \lambda^{-1/2}\tilde{f})$.
By H\"{o}lder's inequality on the terms in $\tilde{F}$ and the smallness of the $d$-dimensional Lebesgue measure of $\Omega^*_R(x_0)$ defined in \eqref{eq0901_01}, we obtain 
\begin{multline}
    \label{eq0107_02}
\left(|Dw|^p + \lambda^{p/2}|w|^p\right)^{1/p}_{Q_R^+} \le N\gamma_0^{1/\nu'}\sum_{j=0}^\infty  2^{-j/4}\left(|\cU|^\nu\right)^{1/\nu}_{C_{2^{j/2}R,R}(X_0)}
\\
+ N\sum_{j=0}^\infty  2^{-j/4}\left(|F|^p\right)^{1/p}_{C_{2^{j/2}R,R}(X_0)}
\\
+ N\gamma_0^{1/\nu'}\sum_{j=0}^\infty  2^{-j/4}\left(|H(D_t^{1/2}u)|^\nu\right)^{1/\nu}_{C_{2^{j/2}R,R}(X_0)}.
\end{multline}
The detailed derivation from \eqref{eq0107_03} to \eqref{eq0107_02} is provided right after the proof of this proposition (see Proof of \eqref{eq0107_02} below).

Now, we set $v=u-w$ in $\bR\times B_{R}^+(x_0)$.
Then, since $u \in H_p^{1/2,1}(\cQ)$ and $w \in H_{p}^{1/2,1}(\bR \times \Pi_R)$, it follows that $v \in H_p^{1/2,1}(\bR \times B_R^+(x_0))$.
In addition, because $u$ satisfies \eqref{eq of u on half space} with the conormal derivative boundary condition on $\bR\times (\{x_1=0\}\cap \overline{B_R(x_0)})$, and $w$ satisfies \eqref{eq0107_01} with the conormal derivative boundary condition on $\bR \times \{(0,x'): x' \in \bR^{d-1}\}$, the function $v$, as an element of $H_p^{1/2,1}(\bR \times B_{R/2}^+)$ (noting that $B_{R/2}^+ \subset B_R^+(x_0)$), satisfies (see Definition \ref{def0814_2} and Remark \ref{rem1213_1})
\begin{equation*}
\overline{\cP}_\lambda v=0 \quad \text{in} \,\,\bR\times B_{R/2}^+
\end{equation*}
with the conormal derivative boundary condition on $\bR\times\Gamma_{R/2}$.

We now define $\cW$ and $\cV$ on $C_R(X_0)$ as follows:
\[
\cW = \left\{
\begin{aligned}
(Dw, \sqrt{\lambda}w) \quad &\text{for} \,\, C_R(X_0) \cap \{(t,x_1,x') \in \bR \times \bR^d : x_1 > 0\},
\\
\cU \quad &\text{for} \,\, C_R(X_0) \cap \{(t,x_1,x') \in \bR \times \bR^d : x_1 \leq 0\},
\end{aligned}
\right.
\]
\[
\cV = \left\{
\begin{aligned}
(Dv, \sqrt{\lambda}v) \quad &\text{for} \,\, C_R(X_0) \cap \{(t,x_1,x') \in \bR \times \bR^d : x_1 > 0\},
\\
(0, 0) \quad &\text{for} \,\, C_R(X_0) \cap \{(t,x_1,x') \in \bR \times \bR^d : x_1 \leq 0\}.
\end{aligned}
\right.
\]
Because $u = w+v$ on $\bR \times B_R^+(x_0)$, we see that $\cU = \cW + \cV$ on $C_R(X_0)$.
Moreover, by H\"older's inequality and \eqref{eq0107_02},
\[
(|\cW|^p)_{C_R(X_0)}^{1/p}
= 
\left(\frac{1}{|C_R(X_0)|}\int_{Q_R^+(X_0)}|\cW|^p+\frac{1}{|C_R(X_0)|}\int_{C_R^*(X_0)}|\cU|^p\right)^{1/p}
\]
\[
\le N(|\cW|^p)_{Q_R^+}^{1/p}+N(I_{C_R^*(X_0)})^{1/\nu'}_{C_R(X_0)}(|\cU|^\nu)^{1/\nu}_{C_R(X_0)}
\]
\[
\le N(|\cW|^p)_{Q_R^+}^{1/p}+N\gamma_0^{1/\nu'}(|\cU|^\nu)^{1/\nu}_{C_R(X_0)}
\]
\[
\le N\gamma_0^{1/\nu'}\sum_{j=0}^\infty  2^{-j/4}\left(|\cU|^\nu\right)^{1/\nu}_{C_{2^{j/2}R,R}(X_0)} + N\sum_{j=0}^\infty  2^{-j/4}\left(|F|^p\right)^{1/p}_{C_{2^{j/2}R,R}(X_0)}
\]
\[
+ N\gamma_0^{1/\nu'}\sum_{j=0}^\infty  2^{-j/4}\left(|H(D_t^{1/2}u)|^\nu\right)^{1/\nu}_{C_{2^{j/2}R,R}(X_0)},
\]
where $C_R^*(X_0)=(-R^2,R^2)\times \Omega_R^*(x_0)$. This proves the estimate \eqref{eq0108_01}.

Next, we apply Lemma \ref{lem half cono local lambda} (also see Remark \ref{rem0903_1}) to $v$ to derive (noting that $\cV = \cU - \cW$ in $Q_{R/2}^+$)
\[
(|\cV|^{p_{1}})^{1/p_{1}}_{Q^+_{R/8}}
\le N(|\cV|^p)^{1/p}_{Q^+_{R/2}} 
\le N(|\cU|^p)^{1/p}_{Q^+_{R/2}}+N(|\cW|^p)^{1/p}_{Q^+_{R/2}}.  
\]
Then, by exploiting H\"older's inequality and \eqref{eq0107_02}, we obtain that 
\[
(|\cV|^{p_{1}})^{1/p_{1}}_{Q^+_{R/8}}\le N(|\cU|^p)^{1/p}_{C_R(X_0)} + N\gamma_0^{1/\nu'}\sum_{j=0}^\infty  2^{-j/4}\left(|\cU|^\nu\right)^{1/\nu}_{C_{2^{j/2}R,R}(X_0)}
\]
\[
+ N\sum_{j=0}^\infty  2^{-j/4}\left(|F|^p\right)^{1/p}_{C_{2^{j/2}R,R}(X_0)} + N\gamma_0^{1/\nu'}\sum_{j=0}^\infty  2^{-j/4}\left(|H(D_t^{1/2}u)|^\nu\right)^{1/\nu}_{C_{2^{j/2}R,R}(X_0)}.   
\]
From this with the fact that $\cV = (0, 0)$ on $(-R^2,R^2) \times \Omega^*_R(x_0)$, we see that \eqref{eq0108_02} holds.
The proposition is proved.
\end{proof}

\begin{proof}[Proof of \eqref{eq0107_02}]
Recall the definition of $\Omega^*_R(x_0)$ in \eqref{eq0901_01}, which satisfies
\[
\Omega_R^*(x_0) = \Omega_R(x_0)\cap \{x:-2\gamma_0 R<x_1<0\},
\]
where
\begin{equation}
							\label{eq0107_04}
|\Omega_R(x_0)\cap \{x:-2\gamma_0 R<x_1<0\}| \leq N \gamma R^d.
\end{equation}
Also observe that
\[
|B_R^+|\le N|B_R^+(x_0)|, \quad |\Omega_R(x_0)|\le N|B_R^+|.
\]
By the definition of $\tilde{h}$ in Lemma \ref{lem0901_1} (in fact, $\tilde{h}$ in \eqref{eq0107_01} is the one extended as zero on $\bR \times (B_R^+ \setminus B_R^+(x_0)$), along with the above observations, we have
\[
(|\tilde{h}|^p)^{1/p}_{Q_{2^{j/2} R,R}^+} \le N(|{h}|^p)^{1/p}_{C_{2^{j/2}R,R}(X_0)} + N(|{h}I_{\{-2\gamma_0 R<x_1<0\}}|^p )^{1/p}_{C_{2^{j/2} R,R}(X_0)}
\]
\[
+N(|{(H(D_t^{1/2}u))}I_{\{-2\gamma_0 R<x_1<0\}}|^p )^{1/p}_{C_{2^{j/2} R,R}(X_0)}
\]
\[
\le N(|{h}|^p)^{1/p}_{C_{2^{j/2} R,R}(X_0)} + N\gamma_0^{1/\nu'}(|H(D_t^{1/2}u)|^\nu )^{1/\nu}_{C_{2^{j/2} R,R}(X_0)},
\]
where we used H\"older's inequality and \eqref{eq0107_04} to get the last inequality.
We follow the same calculation as above to obtain
\begin{multline}
							\label{eq0107_05}
(|\tilde{g}|^p)^{1/p}_{Q_{2^{j/2} R,R}^+}
\le N(|g|^p)^{1/p}_{C_{2^{j/2} R,R}(X_0)} + N(|a_{ij}-\bar{a}_{ij}|^p |Du|^p)^{1/p}_{C_{2^{j/2} R,R}(X_0)}
\\
+N\gamma_0^{1/\nu'}(|{Du}|^\nu )^{1/\nu}_{C_{2^{j/2} R,R}(X_0)},
\end{multline}
where by H\"older's inequality, $|a_{ij}|\le \delta^{-1}$, and 
Assumption \ref{assum rei coeffi} (see \eqref{coeffi bmo} and Remark \ref{rem0902_1}), we see that 
\[
(|a_{ij}-\bar{a}_{ij}|^p|Du|^p)_{C_{2^{j/2}R,R}(X_0)}^{1/p}
\le (|a_{ij}-\bar{a}_{ij}|^{\nu'})_{C_{2^{j/2}R,R}(X_0)}^{1/\nu'}
(|Du|^\nu)_{C_{2^{j/2}R,R}(X_0)}^{1/\nu}
\]
\[
\le N(|a_{ij}-\bar{a}_{ij}|)_{C_{2^{j/2}R,R}(X_0)}^{1/\nu'} (|Du|^\nu)_{C_{2^{j/2}R,R}(X_0)}^{1/\nu} \le
N\gamma_0^{1/\nu'} (|Du|^\nu)_{C_{2^{j/2}R,R}(X_0)}^{1/\nu}.
\]
Hence, the second term on the right-hand side of \eqref{eq0107_05} is bounded by the last term.
In the same manner, we have 
\[
 (|\tilde{f}|^p)^{1/p}_{Q_{2^{j/2} R,R}^+}
\le N(|{f}|^p )^{1/p}_{C_{2^{j/2} R,R}(X_0)}
  +N{\lambda}\gamma^{1/\nu'}(|{u}|^\nu )^{1/\nu}_{C_{2^{j/2} R,R}(X_0)}.
\]
Combining the above estimates for $\tilde{h}$, $\tilde{g}$, and $\tilde{f}$ with the inequality \eqref{eq0107_03}, we finally obtain \eqref{eq0107_02}.
\end{proof}

\subsection{Half-derivative estimates for the heat equation on a Reifenberg-flat domain}

In this subsection, we obtain estimates for the heat equation on a Reifenberg-flat domain. 
\begin{lemma}
							\label{lem0906_1}
Let $p \in (1,\infty)$.
There exists $N = N(d,p)$ such that, for any $\lambda \geq 0$,
\[
\||D_t^{1/2}u| + |Du| + \sqrt{\lambda}|u|\|_{L_\infty(Q_R)} \leq N \left(|D_t^{1/2}u|^p + |Du|^p + \lambda^{p/2}|u|^p\right)^{1/p}_{Q_{4R}},
\]
provided that $u \in H_p^{1/2,1}(\bR \times B_{4R})$, and
\[
u_t - \Delta u + \lambda u = 0
\]
in $\bR \times B_{4R}$.
\end{lemma}

\begin{proof}
Thanks to scaling, we assume $R = 1$.
Then, the lemma follows from \cite[Lemma 5.5]{MR4920684} for $p=2$.
Moreover, the same proof applies for $p \in (1,\infty)$. Recall also that the $L_\infty$-norm can be bounded by the sum of the Lipschitz semi-norm and the $L_p$-norm.
\end{proof}

\begin{lemma}
					\label{lem half cono local lambda laplacian}
Let $p \in (1,\infty)$, $\lambda \geq 0$, and $u\in H_p^{1/2,1}(\bR\times B_{4R}^+)$ satisfy
\[
u_t - \Delta u + \lambda u = 0 \quad \text{in} \,\, \bR \times B_{4R}^+
\]
with the conormal derivative boundary condition on $\bR\times \Gamma_{4R}$.
Then, we have
\[
\||D_t^{1/2}u| + |Du| + \sqrt{\lambda}|u|\|_{L_\infty(Q_R^+)} \leq N \left(|D_t^{1/2}u|^p + |Du|^p + \lambda^{p/2}|u|^p\right)^{1/p}_{Q_{4R}^+},
\]
where $N=N(d,p)$.
\end{lemma}
\begin{proof}
We take the even extension of $u$ with respect to the $x_1$-variable as follows:
\[
\bar{u}(t,x_1,x')
=
\begin{cases}
  u(t,x_1,x')  &\text{for} \quad (t,x_1,x')\in \bR\times B_{4R}^+,
  \\
  u(t,-x_1,x')  &\text{for} \quad (t,x_1,x')\in \bR\times B_{4R}^-.
\end{cases}
\]
Then, $\bar{u}$ belongs to $H_p^{1/2,1}(\bR \times B_{4R})$.
In particular, we have
\[
D_t^{1/2}\bar{u} = D_t^{1/2}u(t,x_1,x')I_{x_1>0} + D_t^{1/2}u(t,-x_1,x')I_{x_1<0},
\]
see \cite[Remark 3.5]{MR4920684}.
Moreover, due to the boundary condition, $\bar{u}$ satisfies $\bar{u}_t-\Delta \bar{u}+\lambda \bar{u} = 0$ in $\bR \times B_{4R}$.
Thus, the desired inequality follows from Lemma \ref{lem0906_1}.
\end{proof}

Now, we are ready to establish key estimates for solutions to  the heat equation on a Reifenberg-flat domain.
The proof is almost identical to that of Proposition \ref{prop reifen}, except that estimates for $D_t^{1/2}w$ and $D_t^{1/2}v$ (where $u=w+v$) are additionally obtained, and Lemmas \ref{lem0906_1} and \ref{lem half cono local lambda laplacian} are used in place of Lemma \ref{lem half cono local lambda}.


\begin{proposition}
							\label{prop reifen laplacian} 
Let $\lambda \in (0,\infty)$, $p \in (1,\infty)$, $\nu \in(p,\infty)$, $\nu'=\nu p/(\nu-p)$, $R\in (0,R_0]$.
Assume that $u \in H_\nu^{1/2,1}(\cQ) \cap H_p^{1/2,1}(\cQ)$ satisfies
\begin{equation*}
u_t - \Delta u + \lambda u = D_t^{1/2} h +D_ig_i+f \quad \text{in} \,\, \bR\times \Omega
\end{equation*}
with the conormal derivative boundary condition on $\bR\times {\partial\Omega}$, where
$h,g_i,f \in L_\nu(\cQ) \cap L_p(\cQ)$.
Denote $F$ as in \eqref{eq0108_03} and
\[
U = (D_t^{1/2}u, Du, \sqrt{\lambda}u).
\]
Then, under Assumption \ref{assum rei coeffi} ($\gamma$), there exists $N=N(d,\delta,p,\nu)$ such that the following hold true. 
\begin{itemize}
\item[(i)] 
For $X_0=(t_0,x_0)\in \bR\times \Omega$ with $\operatorname{dist}(x_0,\partial\Omega)\ge R$, there exist vector-valued functions $V$ and $W$ in $Q_R(X_0)$ such that $U = W + V$ in $Q_R(X_0)$, where $W$ and $W$ satisfy 
\begin{align*}
(|W|^p)_{Q_R(X_0)}^{1/p} &\le N\gamma_0^{1/\nu'}\sum_{j=0}^\infty2^{-j/4}(|U|^\nu)^{1/\nu}_{Q_{2^{j/2}R,R}(X_0)}\\
&\quad +N\sum_{j=0}^\infty2^{-j/4}(|F|^p)^{1/p}_{Q_{2^{j/2}R,R}(X_0)},
\end{align*}
and 
\begin{align*}
\|V\|_{L_\infty(Q_{R/8}(X_0))}
&\le  N\gamma_0^{1/\nu'}\sum_{j=0}^\infty2^{-j/4}(|\cU|^\nu)^{1/\nu}_{Q_{2^{j/2}R,R}(X_0)}\\
&\quad +N\sum_{j=0}^\infty2^{-j/4}(|F|^p)^{1/p}_{Q_{2^{j/2}R,R}(X_0)} +N(|U|^p)^{1/p}_{Q_R(X_0)}.
\end{align*}

\item[(ii)]
For $X_0=(t_0,x_0)\in \bR\times \partial\Omega$,
there exist vector-valued functions $W$ and $V$ in $C_R(X_0)$ such that $U = W + V$ in $C_R(X_0)$, where $W$ and $V$ satisfy 
\begin{equation*}
\begin{aligned}
(|W|^p)_{C_R(X_0)}^{1/p}
&\le N\gamma_0^{1/\nu'}\sum_{j=0}^\infty  2^{-j/4}\left(|U|^\nu\right)^{1/\nu}_{C_{2^{j/2}R,R}(X_0) }
\\ 
&+ N\sum_{j=0}^\infty  2^{-j/4}\left(|F|^p\right)^{1/p}_{C_{2^{j/2}R,R}(X_0) }
\\
+ N\gamma_0^{1/\nu'}&\sum_{j=0}^\infty  2^{-j/4}\left(|H(D_t^{1/2}u)|^\nu\right)^{1/\nu}_{C_{2^{j/2}R,R}(X_0) },
\end{aligned}
\end{equation*}
and
\begin{equation*}
\begin{aligned}
\|V\|_{L_\infty(C_{R/8}(X_0))} \le N\gamma_0^{1/\nu'}\sum_{j=0}^\infty  2^{-j/4}\left(|U|^\nu\right)^{1/\nu}_{C_{2^{j/2}R,R}(X_0)}
\\
+N\sum_{j=0}^\infty  2^{-j/4}\left(|F|^p\right)^{1/p}_{C_{2^{j/2}R,R}(X_0)} + N(|U|^p)^{1/p}_{C_R(X_0)}
\\
+ N\gamma_0^{1/\nu'}\sum_{j=0}^\infty  2^{-j/4}\left(|H(D_t^{1/2}u)|^\nu\right)^{1/\nu}_{C_{2^{j/2}R,R}(X_0)}.   
\end{aligned}
\end{equation*}
\end{itemize}
\end{proposition}

\begin{proof}
Since one can follow the proof of Proposition \ref{prop reifen}, we omit the details.
A few remarks are in order.
Compared to Proposition \ref{prop reifen}, the proof of this proposition requires estimates for $D_t^{1/2}w$ and $D_t^{1/2}v$.
In fact, one can include $D_t^{1/2}w$ on the left-hand side of the inequalities in \eqref{eq0107_03} and \eqref{eq0107_02},  because the estimate in Lemma \ref{lem0106_1} accounts for the corresponding term on the left-hand side.
However, when the coefficients are $\bar{a}_{ij}(t)$, the $L_\infty$-estimate of $D_t^{1/2}v$ is not available. 
For this reason, we have not included the estimates for $D_t^{1/2}v$, nor for $D_t^{1/2}w$.
In contrast, when the Laplace operator is present, $L_\infty$-estimates of $v$ are available (see Lemmas \ref{lem0906_1} and \ref{lem half cono local lambda laplacian}).
Thus, by following the same argument as in the proof of Proposition \ref{prop reifen}, one can also derive estimates for $D_t^{1/2}v$.
\end{proof}

\section{Proofs of main results}\label{sec proof of main}

\subsection{Proof of Theorem \ref{conormal entire time}: Level set argument}

For $X_0=(t_0,x_0)\in \bR\times\bR^d$ and a function $f$ in $L_{1,\text{loc}}(\bR\times\bR^d)$, we define
$$
\cM f(X_0)=\sup_{Q_r(t,x)\ni X_0}\dashint_{Q_{r}(t,x)}|f|\,dy\,ds
$$
and 
$$
\cS\cM f(X_0)=\sup_{Q_{r_1,r_2}(t,x)\ni X_0}\dashint_{Q_{r_1,r_2}(t,x)}|f|\,dy\,ds.
$$

We are now ready to give
\begin{proof}[Proof of Theorem \ref{conormal entire time}] We first establish the a priori estimate and after that, we prove the existence and uniqueness of solutions first in the unweighted case and then in the weighted mixed-norm case.
	
{\bf A priori estimate}. We first impose the additional condition that $u$ is supported on $Q_{\gamma_0^{1/(d+2)}R_0}(Y_0)$ for some $Y_0=(s_0,y_0)\in \bR\times \overline\Omega$. Here, $\gamma_0$ is the constant from Assumption \ref{assum rei coeffi} which will be specified later. Also, assume that $p=q$. By the self-improving property of $A_p$-weights, we have $\omega_1,\omega_2\in A_{p/\nu}$ for a sufficiently small $\nu\in (1,p)$. Now we take $p_{0}=(\nu+1)/2\in (1,\nu)$. By the reverse H\"older inequality, for any $Q_r(Z_0)\subset \bR^{d+1}$ and a measurable set $E\subset \bR^{d+1}$, we have
\begin{equation}
\label{eq5.17}
\frac{1}{N_0}\left(\frac{|E\cap Q_r(Z_0)|}{|Q_r(Z_0)|}\right)^p \le\frac{\omega(E\cap Q_r(Z_0))}{\omega(Q_r(Z_0))}\le N_0\left(\frac{|E\cap Q_r(Z_0)|}{|Q_r(Z_0)|}\right)^\beta
\end{equation}
 for some small $\beta\in (0,1]$  and $N_0$ depending only on $d$, $p$, and $K$. See \cite[Lemma 4.3]{MR3947859} for further details on the derivation of the above inequalities.
 Note that in the unweighted case, we just take $\beta=1$. Then we fix  a sufficiently large $p_{1}\in (p,\infty)$ depending only on $d$, $p$, and $K$ so that $\beta>p/p_{1}$.

Let $X_0=(t_0,x_0)\in \bR\times \overline\Omega$.
By Proposition \ref{prop reifen} with $p_{0}$ in place of $p$, it is easily seen that when $R\le R_0$, there exist vector-valued functions $\cW$ and $\cV$ in $C_{R/16}(X_0)$ such that $\cU = \cW + \cV$ in $C_{R/16}(X_0)$, where $\cW$ and $\cV$ satisfy 
\begin{equation}
						\label{eq4.49_0}
\begin{aligned}
&\left(|\cW|^{p_{0}}\right)_{C_{R/16}(X_0)}^{1/p_{0}}
\le N\gamma_0^{1/\nu'}\sum_{j=0}^\infty2^{-j/4}\left(|\cU|^\nu\right)^{1/\nu}_{C_{2^{j/2}R,R}(X_0)}\\
&\quad + N\sum_{j=0}^\infty2^{-j/4}\left(|F|^{p_{0}}\right)^{1/p_{0}}_{C_{2^{j/2}R,R}(X_0)}
+N\gamma_0^{1/\nu'}\sum_{j=0}^\infty2^{-j/4}\left(|H(D_t^{1/2}u)|^\nu\right)^{1/\nu}_{C_{2^{j/2}R,R}(X_0)},
\end{aligned}
\end{equation}
and
\begin{equation}
	\label{eq4.50_0}
	\begin{aligned}
&\left(|\cV|^{p_{1}}\right)_{C_{R/128}(X_0)}^{1/p_{1}} \le N\gamma_0^{1/\nu'}\sum_{j=0}^\infty2^{-j/4}\left(|\cU|^\nu\right)^{1/\nu}_{C_{2^{j/2}R,R}(X_0)} \\
&\quad 
+N\sum_{j=0}^\infty2^{-j/4}\left(|F|^{p_{0}}\right)^{1/p_{0}}_{C_{2^{j/2}R,R}(X_0)} +N\left(|\cU|^{p_{0}}\right)^{1/p_{0}}_{C_{R}(X_0)}\\
&\quad +N\gamma_0^{1/\nu'}\sum_{j=0}^\infty2^{-j/4}\left(|H(D_t^{1/2}u)|^\nu\right)^{1/\nu}_{C_{2^{j/2}R,R}(X_0)},
\end{aligned}
\end{equation}
where $\nu'$ satisfies $1/p_{0}=1/\nu+1/\nu'$.
Using the fact that 
\[
|\Omega_r(z)|\ge N(d) r^d, \quad \forall z\in \overline{\Omega}, \quad \forall r\in(0,R_0],
\]
we obtain the following estimates from \eqref{eq4.49_0} and \eqref{eq4.50_0}:
\begin{equation}
						\label{eq4.49}
\begin{aligned}
&\left(I_{\cQ}|\cW|^{p_{0}}\right)_{Q_{R/16}(X_0)}^{1/p_{0}}
\le N\gamma_0^{1/\nu'}\sum_{j=0}^\infty2^{-j/4}\left(I_{\cQ}|\cU|^\nu\right)^{1/\nu}_{Q_{2^{j/2}R,R}(X_0)}\\
&\quad + N\sum_{j=0}^\infty2^{-j/4}\left(I_{\cQ}|F|^{p_{0}}\right)^{1/p_{0}}_{Q_{2^{j/2}R,R}(X_0)}\\
&\quad +N\gamma_0^{1/\nu'}\sum_{j=0}^\infty2^{-j/4}\left(I_{\cQ}|H(D_t^{1/2}u)|^\nu\right)^{1/\nu}_{Q_{2^{j/2}R,R}(X_0)}
\end{aligned}
\end{equation}
and
\begin{equation}
	\label{eq4.50}
	\begin{aligned}
&\left(I_{\cQ}|\cV|^{p_{1}}\right)_{Q_{R/128}(X_0)}^{1/p_{1}} \le N\gamma_0^{1/\nu'}\sum_{j=0}^\infty2^{-j/4}\left(I_{\cQ}|\cU|^\nu\right)^{1/\nu}_{Q_{2^{j/2}R,R}(X_0)} \\
&\quad 
+N\sum_{j=0}^\infty2^{-j/4}\left(I_{\cQ}|F|^{p_{0}}\right)^{1/p_{0}}_{Q_{2^{j/2}R,R}(X_0)} +N\left(I_{\cQ}|\cU|^{p_{0}}\right)^{1/p_{0}}_{Q_{R}(X_0)}\\
&\quad +N\gamma_0^{1/\nu'}\sum_{j=0}^\infty2^{-j/4}\left(I_{\cQ}|H(D_t^{1/2}u)|^\nu\right)^{1/\nu}_{Q_{2^{j/2}R,R}(X_0)}.
\end{aligned}
\end{equation}

Now for $\mathsf{s}>0$ and $\gamma\in (0,1)$, we denote
$$
\cA(\mathbf{s})=\{(t,x)\in \bR\times \overline \Omega\,:\, |\cU|>\mathsf{s}\}
$$
and 
\begin{align*}
	&\cB(\mathbf{s})=\{(t,x)\in \bR\times \overline \Omega\,:\, \gamma^{-\frac 1 {p_1}}(\cM (I_{\cQ}|\cU|^{p_0})(t,x))^{\frac 1{p_0}}\\
	&\quad+\gamma^{-\frac 1 {p_0}}\gamma_0^{\frac 1 {\nu'}}(\cS\cM (I_{\cQ}|\cU|^{\nu})(t,x))^{\frac 1{\nu}}     +\gamma^{-\frac 1 {p_0}}\gamma_0^{\frac 1 {\nu'}}(\cS\cM (I_{\cQ}|H(D_t^{1/2}u)|^\nu)(t,x))^{\frac 1{\nu}}\\
	& \quad    +\gamma^{-\frac 1 {p_0}}(\cS\cM (I_{\cQ}|F|^{p_0})(t,x))^{\frac 1{p_0}}>\mathsf{s}\}.
\end{align*}
We claim that for a sufficiently large $\kappa=\kappa(d,\delta,p_0,p_1)>0$, the following hold: for $0<R\le R_0/128$, $X_0=(t_0,x_0)\in \bR\times \overline\Omega$, and $\mathsf{s}>0$, if 
\begin{equation}
	\label{eq5.16}
\omega(Q_R(X_0)\cap \cA(\kappa \mathsf{s}))\ge N_0\gamma^\beta \omega(Q_R(X_0)),
\end{equation}
then
$$
C_{R}(X_0)\subset \cB(\mathsf{s}).
$$
Here, $N_0$ is the constant from the inequality \eqref{eq5.17}.
Indeed, by dividing the equation by $\mathsf{s}$,  we may assume that $\mathsf{s}=1$. Suppose that $C_{R}(X_0)\not\subset \cB(\mathsf{s})$ so that there is a point $(s,y)\in C_R(X_0)$ such that 
\begin{align*}
	&\gamma^{-\frac 1 {p_1}}(\cM (I_{\cQ}|\cU|^{p_0})(s,y))^{\frac 1{p_0}}
	+\gamma^{-\frac 1 {p_0}}\gamma_0^{\frac 1 {\nu'}}\cS\cM (I_{\cQ}|\cU|^{\nu})(s,y))^{\frac 1{\nu}}\\
	&\quad +\gamma^{-\frac 1 {p_0}}\gamma_0^{\frac 1 {\nu'}}(\cS\cM (I_{\cQ}|H(D_t^{1/2}u)|^\nu)(s,y))^{\frac 1{\nu}}
	+\gamma^{-\frac 1 {p_0}}(\cS\cM (I_{\cQ}|F|^{p_0})(s,y))^{\frac 1{p_0}}\le 1.
\end{align*}
This, together with \eqref{eq4.49} and \eqref{eq4.50} with $128R$ in place of $R$, implies that
$$\left(I_{\cQ}|\cW|^{p_{0}}\right)_{Q_{R}(X_0)}^{1/p_{0}}
\le N_{1}\gamma^{1/p_0}
$$
and
$$
\left(I_{\cQ}|\cV|^{p_{1}}\right)_{Q_{R}(X_0)}^{1/p_{1}} 
\le N_{1}(\gamma^{1/p_0}+\gamma^{1/p_1})\le N_{1}\gamma^{1/p_1},
$$
where $N_{1}$ depends only on $d$, $\delta$, $p_0$, and $p_1$. 
Then, following the proof of \cite[Lemma 5.2]{MR4030286}, we show that for sufficiently large $\kappa$ depending on the same set of parameters, we have
\begin{equation}
\label{eq0825_1}
|Q_R(X_0)\cap \cA(\kappa )|
< \gamma |Q_R(X_0)|.    
\end{equation}
In order to establish this, for $0<M<\kappa$, by Chebyshev's inequality, we have \[
|Q_R(X_0)\cap \cA(\kappa)|
=|\{(t,x)\in Q_R(X_0):(I_\cQ|\cU|)(t,x)>\kappa\}|
\]
\[\le \left|\{(t,x)\in Q_R(X_0):(I_\cQ|\cW|)(t,x)>\kappa -M \}\right|
\]
\[
+\left|\{(t,x)\in Q_R(X_0):(I_\cQ|\cV|)(t,x)>M \}\right|\]
\[
\leq (\kappa-M)^{-p_0} \int_{Q_{R}(X_0)} I_{\cQ}|\cW|^{p_0} + M^{-p_1} \int_{Q_{R}(X_0)} I_{\cQ}|\cV|^{p_1}\]
\[\leq \frac{N_1^{p_0} \gamma |Q_R(X_0)|}{(\kappa - M)^{p_0}} + \frac{N_1^{p_1}\gamma|Q_{R}(X_0)|}{M^{p_1}}\le \gamma |Q_R(X_0)|\]
provided that we choose a sufficiently large $M$ depending only on $d$, $\delta$, $p_0$, and $p_1$ so that
$$
 (N_1/M)^{p_1} < 1/2,
$$
and then choose a $\kappa$ depending only on $d$, $\delta$, $p_0$, and $p_1$ so that
$$
N_1^{p_0}/(\kappa-M)^{p_0} < 1/2.
$$
This yields \eqref{eq0825_1}.
We note that the inequality \eqref{eq0825_1} contradicts \eqref{eq5.16} in view of \eqref{eq5.17}. The claim is proved.

We also claim that for any $X_0\in \bR\times \overline{\Omega}$ and $R\ge R_0/128$, we have 
\begin{equation}
    \label{eq0825_2}
    \omega(\cA(\kappa \mathbf{s})\cap Q_R(X_0))< N_0\gamma^\beta \omega (Q_R(X_0)),
\end{equation}
provided that 
\[
\mathbf{s}>\mathbf{s}_0:= \frac{N_2}{N_0\gamma^\beta \omega(Q_{R_0/2}(Y_0))^{\nu/p}}\|\cU\|_{L_{{p/\nu},\omega}(\cQ)},
\]
where $N_2=N_2(d,K)$.
Since $\operatorname{supp}u\subset Q_{\gamma_0^{1/(d+2)}R_0}(Y_0)$, it is sufficient to consider the case that $Q_R(X_0)\cap Q_{\gamma_0^{1/(d+2)}R_0}(Y_0)\neq \emptyset$. In this case, we observe that 
\[
Q_{R_0/2}(Y_0)\subset Q_{256R}(X_0).
\]
Using this, Chebyshev's inequality, H\"older's inequality, and the doubling property of $A_p$-weights, we obtain 
\[
\omega(\cA(\kappa\mathbf{s})\cap Q_R(X_0))
\le \frac{1}{\kappa \mathbf{s}}\int_{Q_{256R}(X_0)}I_{\cQ}|\cU|\omega
\]
\[
\le \frac{1}{\kappa \mathbf{s}}\omega(Q_{256R}(X_0))^{1-\nu/p}\|\cU\|_{L_{p/\nu,\omega}(\cQ)}
\le  \frac{N}{\kappa \mathbf{s}}\frac{\omega(Q_{R}(X_0))}{\omega(Q_{R_0/2}(Y_0))^{\nu/p}}\|\cU\|_{L_{p/\nu,\omega}(\cQ)}.
\]
This verifies the claim \eqref{eq0825_2}.
By combining the above two claims and the ``crawling of ink spots'' lemma due to Krylov and Safonov \cite{MR563790}, we conclude that
\begin{equation}
						\label{eq5.32}
\omega(\cA(\kappa \mathsf{s}))\le N\gamma^\beta \omega(\cB(\mathsf{s})),\quad \mathbf{s}>\mathbf{s_0}.
\end{equation}
Here, we take $\gamma$ such that $N_0\gamma^\beta<1$. For a detailed proof, we refer to \cite{MR3467697,MR3947859}. By utilizing \eqref{eq5.32}, we have
\begin{equation}
	\label{eq0825_4}
\begin{aligned}
	&\|\cU\|_{L_{p,\omega}(\cQ)}^p=p\int_0^\infty \omega(\cA(\mathsf{s}))\mathsf{s}^{p-1}\,d\mathsf{s}
	=p\kappa^p \int_0^\infty \omega(\cA(\kappa\mathsf{s}))\mathsf{s}^{p-1}\,d\mathsf{s}\\
	&\le N\int_0^{\mathbf{s}_0}\omega(\cA(\kappa\mathbf{s}))\mathbf{s}^{p-1}\,d\mathbf{s}+N\gamma^\beta \int_0^\infty \omega(\cB(\mathsf{s}))\mathsf{s}^{p-1}\,d\mathsf{s}\\
    &=:I_1+I_2.
    \end{aligned}
    \end{equation}
    We restrict $\gamma_0$ such that $\gamma_0^{1/(d+2)}\le 1/2$.
    By utilizing Chebyshev's inequality, H\"older's inequality, and \eqref{eq5.17}, we have  
    \begin{equation}
        \label{eq0825_5}
        \begin{aligned}
        I_1&\le N\left(\int_0^{\mathsf{s}_0}\mathsf{s}^{p(1-               1/\nu)-1}\,d\mathsf{s}\right)\|\cU\|_{L_{p/\nu,\omega}(\cQ)}^{p/\nu}
        \\
        &\le N\gamma^{\beta p(1/\nu-1)}\omega(Q_{R_0/2}(Y_0))^{1-\nu}\|\cU\|_{L_{p/\nu,\omega}(\cQ)}^{p}
        \\
        &\le N\gamma^{\beta p(1/\nu-1)}\left(\frac{\omega(Q_{\gamma_0^{1/(d+2)}R_0}(Y_0))}{\omega(Q_{R_0/2}(Y_0))}\right)^{\nu-1}\|\cU\|_{L_{p,\omega}(\cQ)}^{p}
        \\
        &\le N\gamma^{\beta p(1/\nu-1)}\gamma_0^{\beta(\nu-1)}\|\cU\|_{L_{p,\omega}(\cQ)}^{p}.
        \end{aligned}
        \end{equation}
        By \eqref{eq5.32} and the weighted Hardy-Littlewood maximal function theorem, we obtain
    \begin{equation}
    \label{eq0825_6}
    \begin{aligned}
    I_2
	&\le N\gamma^\beta \int_0^\infty\omega\big( \big\{  \gamma^{-\frac 1 {p_1}}(\cM (I_{\cQ}|\cU|^{p_0})(t,x))^{\frac 1{p_0}}>\mathsf{s}/4 \big\}\big)\mathsf{s}^{p-1}\,d\mathsf{s}\\
	&\quad +N\gamma^\beta \int_0^\infty\omega\big( \gamma^{-\frac 1 {p_0}}\gamma_0^{\frac 1 {\nu'}}(\cS\cM (I_{\cQ}|\cU|^{\nu})(t,x))^{\frac 1{\nu}}>\mathsf{s}/4 \big\}\big)\mathsf{s}^{p-1}\,d\mathsf{s}\\
	&\quad +N\gamma^\beta \int_0^\infty\omega\big( \gamma^{-\frac 1 {p_0}}\gamma_0^{\frac 1 {\nu'}}(\cS\cM (I_{\cQ}|H(D_t^{1/2}u|)^{\nu})(t,x))^{\frac 1{\nu}}>\mathsf{s}/4 \big\}\big)\mathsf{s}^{p-1}\,d\mathsf{s}\\
	&\quad +N\gamma^\beta \int_0^\infty\omega\big( \gamma^{-\frac 1 {p_0}}(\cS\cM (I_{\cQ}|F|^{p_0})(t,x))^{\frac 1{p_0}}>\mathsf{s}/4 \big\}\big)\mathsf{s}^{p-1}\,d\mathsf{s}\\
	& \le N(\gamma^{\beta-p/p_1}+\gamma^{\beta-p/p_0}\gamma_0^{p/\nu'})\|\cU\|_{L_{p,\omega}(\cQ)}^p\\
	&\quad +N\gamma^{\beta-p/p_0}\gamma_0^{p/\nu'}\|H(D_t^{1/2}u)\|_{L_{p,\omega}(\cQ)}^p+N\gamma^{\beta-p/p_0}\|F\|_{L_{p,\omega}(\cQ)}^p.
\end{aligned}
\end{equation}
By connecting \eqref{eq0825_4},\eqref{eq0825_5}, and \eqref{eq0825_6}, we reach that 
\begin{multline}
							\label{eq6.32}
\|\cU\|_{L_{p,\omega}(\cQ)}^p\le N(\gamma^{\beta p(1/\nu-1)}\gamma_0^{\beta(\nu-1)}+\gamma^{\beta-p/p_1}+\gamma^{\beta-p/p_0}\gamma_0^{p/\nu'})\|\cU\|_{L_{p,\omega}(\cQ)}^p
\\
+N\gamma^{\beta-p/p_0}\gamma_0^{p/\nu'}\|H(D_t^{1/2}u)\|_{L_{p,\omega}(\cQ)}^p+N\gamma^{\beta-p/p_0}\|F\|_{L_{p,\omega}(\cQ)}^p.
\end{multline}
To estimate the second term on the right-hand side above, we rewrite the equation into
$$
u_t-\Delta u+\lambda u=D_t^{1/2} h+D_i (g_i+(a_{ij}-\delta_{ij})D_j u)+f
$$
with the corresponding conormal boundary condition. Following the proof above by using Proposition \ref{prop reifen laplacian} instead of Proposition \ref{prop reifen} and with $p_1=\infty$, we have
\begin{multline}
							\label{eq5.54}
\|U\|_{L_{p,\omega}(\cQ)}^p	 \le N(\tilde{\gamma}^{\beta p(1/\nu-1)}\gamma_0^{\beta(\nu-1)}+\tilde \gamma^{\beta}+\tilde\gamma^{\beta-p/p_0}\gamma_0^{p/\nu'})\|U\|_{L_{p,\omega}(\cQ)}^p
\\
+N\tilde\gamma^{\beta-p/p_0}\gamma_0^{p/\nu'}\|H(D_t^{1/2}u)\|_{L_{p,\omega}(\cQ)}^p+N\tilde\gamma^{\beta-p/p_0}\|F+|Du|\|_{L_{p,\omega}(\cQ)}^p
\end{multline}
for any $\tilde \gamma\in (0,1)$. By choosing $\tilde \gamma$ sufficiently small and then $\gamma_0$ sufficiently small, we get from \eqref{eq5.54} that
\begin{equation}
	\label{eq6.40}
		\|U\|_{L_{p,\omega}(\cQ)}^p	 \le N\|F+|Du|\|_{L_{p,\omega}(\cQ)}^p.
\end{equation}
Recall the definition of $\cU$ and $U$.
Plugging \eqref{eq6.32} into \eqref{eq6.40} and taking $\gamma$ sufficiently small (recalling that $\beta>p/p_1$) and then $\gamma_0$ sufficiently small, we get the desired a priori estimate when $p=q$ and $u$ is supported on a small cylinder.

To remove the condition that $u$ is supported on a small cylinder, we use a standard partition of unity argument, as in the proof of \cite[Corollary 2.9]{MR4541086}, to get 
\begin{multline*}
	\|D_t^{1/2}u\|_{L_{p,\omega}(\cQ)} + \|Du\|_{L_{p,q,\omega}(\cQ)}+\sqrt{\lambda}\|u\|_{L_{p,\omega}(\cQ)}\leq N \big( \|h\|_{L_{p,\omega}(\cQ)}
\\
+ \|g_i\|_{L_{p,\omega}(\cQ)} + \lambda^{-1/2}\|f\|_{L_{p,\omega}(\cQ)}+\lambda^{-1/2}\|Du\|_{L_{p,\omega}(\cQ)}+\|u\|_{L_{p,\omega}(\cQ)}\big).
\end{multline*}
By taking $\lambda\ge \lambda_0$ sufficiently large, we obtain the a priori estimate when $p=q$. For general $p,q\in (1,\infty)$, we apply the extrapolation theorem in \cite{MR3812104}.

\vspace{.5em}
{\bf Solvability.}
We treat the unweighted case and weighted case separately. 

{\it Case 1: the unweighted case.}
The uniqueness follows from the a priori estimate.
We first consider the case where $a_{ij} = \delta_{ij}$ (i.e., the Laplace operator), the spatial domain $\Omega$ has a smooth boundary, and $f$, $g$, and $h$ are smooth, compactly supported functions on $\mathcal{Q}$.
Then the problem can be viewed as a nondivergence form equation with the zero Neumann boundary condition, so that there is a solution $u\in W^{1,2}_p(\cQ) \subset H^{1/2,1}_p(\cQ)$.
For general $f,g,h\in L_p(\cQ)$, we approximate them by smooth functions with compact support. For general $\Omega$, we approximate them from inside by a sequence of smooth domains, and use the weak compactness to obtain a solution. Finally, for general $a_{ij}$, we apply the method of continuity by using the a priori estimate we just proved.

{\it Case 2: the weighted case.}
From the method of continuity, it suffices to consider $\Delta$.
By the reverse H\"older inequality, one can find a sufficiently large $p_1$, depending only on $d$, $p$, $q$, and $K$, such that the small constants $\varepsilon_1,\varepsilon_2\in (0,1)$ defined by 
\[
1-\frac{p}{p_1}=\frac{1}{1+\varepsilon_1}, \quad 1-\frac{q}{p_1}=\frac{1}{1+\varepsilon_2},\] satisfy 
 $\omega_1^{1+\varepsilon_1}\in A_{p}(\bR^d)$ and $\omega_2^{1+\varepsilon_2}\in A_q(\bR)$. From the doubling property of $A_p$-weights, we have that for any $r\in (0,\infty)$, 
 \begin{equation}
     \label{eq0814_00}
 \int_{B_{2r}} \omega_1^{1+\varepsilon_1}\,dx \le N_1 \int_{B_r} \omega_1^{1+\varepsilon_1}\,dx, \quad
 \int_{-4r^2}^{4r^2} \omega_2^{1+\varepsilon_2}\,dt \le N_1
 \int_{-r^2}^{r^2} \omega_2^{1+\varepsilon_2}\,dt. 
 \end{equation}
 Here, $N_1$ is independent of $r$. Recall that $C_r=(-r^2,r^2)\times \Omega_r$.
 From H\"older's inequality, we see that for $r>0$,
 \begin{equation}
     \label{eq0814_011}
 \|f\|_{L_{p,q,\omega}(C_r)}\le \|\omega_1\|_{L_{1+\varepsilon_1}(B_r)}^{1/p} \|\omega_2\|_{L_{1+\varepsilon_2}(-r^2,r^2)}^{1/p}\cdot\|f\|_{L_{p_1}(C_r)}.
 \end{equation}
 
By mollifying in $t$, we may assume that $D_t^{1/2}h\in L_{p,q,\omega}(\cQ)$. Set $\tilde{f}=f+D_t^{1/2}h\in 
 L_{p,q,\omega}(\cQ)$. 
By the denseness of $C_0^\infty(\cQ)$ in $L_{p,q,\omega}(\cQ)$,  we may further assume that $g_i$ and $\tilde{f}$ are smooth with compact support in $C_R$ for some $R>0$.  Then, from the $L_{p_1}$-solvability which we established, we find $u\in H_{p_1}^{1/2,1}(\cQ)$ such that 
 \begin{equation}
     \label{eq0814_06}
 u_t-\Delta u+\lambda u=D_ig_i+\tilde{f}
 \end{equation}
in $\cQ$ and $u$ satisfies the conormal derivative boundary condition on $\bR\times \partial\Omega$.
Here, $\lambda$ is a sufficiently large constant, which we fix from now on.
By \eqref{eq0814_011}, we have 
\begin{equation}
    \label{eq0814_02}
    \||u|+|Du|\|_{L_{p,q,\omega}(C_{2R})}<\infty.
\end{equation}
For a nonnegative integer $k$, we take a sequence of smooth functions $\eta_k$ such that $\eta_k=0$ in $Q_{2^kR}$ and $1$ outside $Q_{2^{k+1}R}$, and 
\[
|D\eta_k|\le N2^{-k}, \quad |(\eta_k)_t|\le N2^{-2k}.
\]
We observe that $u\eta_k\in H_{p_1}^{1/2,1}(\cQ)$ satisfies 
\[
(u\eta_k)_t-\Delta(u\eta_k)+\lambda u\eta_k
=
u(\eta_k)_t-D_iuD_i\eta_k-D_i(uD_i\eta_k)
\]
in $\cQ$ with the conormal derivative boundary condition on $\bR\times \partial\Omega$. 
Applying the a priori estimate, we obtain that
\[
\| |u\eta_k|+|D(u\eta_k)|\|_{L_{p_1}(\cQ)}
\le N \|u (\eta_k)_t\|_{L_{p_1}(\cQ)}
+N\|u D(\eta_k)\|_{L_{p_1}(\cQ)}
\]
\[+N\|Du D(\eta_k)\|_{L_{p_1}(\cQ)},\]
which yields 
\[
 \||u|+|Du|\|_{L_{p_1}(\cQ\cap (Q_{2^{k+2}R}\setminus Q_{2^{k+1}R}))}\le
\| |u|+|Du|\|_{L_{p_1}(\cQ\setminus Q_{2^{k+1}R})}
\]
\[\le N2^{-k} \||u|+|Du|\|_{L_{p_1}(\cQ\cap (Q_{2^{k+1}R}\setminus Q_{2^kR}))}.
\]
By induction, we reach that for $k=1,2,\ldots$, 
\begin{equation}
    \label{eq0814_03}
     \||u|+|Du|\|_{L_{p_1}(\cQ\cap (Q_{2^{k+1}R}\setminus Q_{2^{k}R}))}\le N 2^{-\frac{k(k-1)}{2}}
\| |u|+|Du|\|_{L_{p_1}(C_{2R})}.
\end{equation}
By H\"older's inequality, \eqref{eq0814_00}, and \eqref{eq0814_03}, it follows that for $k=1,2,\ldots$, 
\[
\||u|+|Du|\|_{L_{p,q,\omega}(\cQ\cap (Q_{2^{k+1}R}\setminus Q_{2^{k}R}))}
\]
\[
\le \|\omega_1\|_{L_{1+\varepsilon_1}(B_{2^{k+1}R})}^{1/p} \|\omega_2\|_{L_{1+\varepsilon_2}(-4^{k+1}R^2,4^{k+1}R^2)}^{1/q}\cdot\||u|+|Du|\|_{L_{p_1}(\cQ\cap (Q_{2^{k+1}R}\setminus Q_{2^{k}R}))}
\]
\[
\le NN_1^{k(\frac{1}{p}+\frac{1}{q})}2^{-\frac{k(k-1)}{2}}
\| |u|+|Du|\|_{L_{p_1}(C_{2R})}.
\]
This estimate together with \eqref{eq0814_02} implies that $u,Du\in L_{p,q,\omega}(\cQ)$. 

We differentiate the equation \eqref{eq0814_06} with respect to $t$. Then, $u_t$ is also a solution to the conormal derivative problem having the right-hand side with $(\tilde{f})_t$ and $(g_i)_t$. By following the above argument, we arrive at $u_t\in L_{p,q,\omega}(\cQ)$. By utilizing interpolation inequalities, we get $D_t^{1/2}u \in L_{p,q,\omega}(\cQ)$, which implies $u\in H_{p,q,\omega}^{1/2,1}(\cQ)$. This concludes the proof.
\end{proof}



\subsection{Proofs of Theorem \ref{conormal -infty to T} and Corollary \ref{conormal bounded interval}}

We prove Theorem \ref{conormal -infty to T} by applying Theorem \ref{conormal entire time}, for which we first need to establish the uniqueness of solutions to \eqref{main eq -infty to T} under the given boundary condition.
To this end, we consider the equation \eqref{eq0908_01} below with $\lambda > 0$, posed in $\Omega_T = (-\infty, T) \times \Omega$.

\begin{lemma}
							\label{lem0906_2}
Let $p \in [2,\infty)$, $T \in (-\infty,\infty]$, $\lambda > 0$, and $a_{ij}$ satisfy only the conditions in \eqref{eq0521_02}.
 Then there exists at most one solution $u \in \cH_p^1\left(\Omega_T\right)$ to 
\begin{equation}
							\label{eq0908_01}
u_t - D_i(a_{ij}D_ju) + \lambda u = D_i g_i + f
\end{equation}
in $\Omega_T$ with the conormal derivative boundary condition on $(-\infty,T) \times \partial\Omega$, where $g_i, f \in L_p\left(\Omega_T\right)$.
\end{lemma}

\begin{proof}
It suffices to prove that $v \equiv 0$, provided that $v \in \cH_p^1\left(\Omega_T\right)$ satisfies
\begin{equation}
							\label{eq0718_02}
v_t - D_i(a_{ij} D_j v) + \lambda v = 0\end{equation}
in $\Omega_T$ with the conormal derivative boundary condition on $(-\infty,T) \times \partial \Omega$.
Set
\begin{equation*}
v^{(\varepsilon)}(t,x) = \int_{-\infty}^T v(s,x) \eta_\varepsilon(t-s) \, ds,
\end{equation*}
where
\[
\eta_\varepsilon(t) = \varepsilon^{-1}\eta(t/\varepsilon), \quad \eta(t) \in C_0^\infty(\bR), \quad 0 \leq \eta(t) \leq 1, \quad \int_\bR \eta(t) \, dt = 1,
\]
\[
\eta(t) = 0 \quad \text{for}\,\, t \leq 0.
\]
Note that $v^{(\varepsilon)} \in L_p(\Omega_T)$,
\[
\partial_t v^{(\varepsilon)} = \int_{-\infty}^T v(s,x) \varepsilon^{-2}\eta_t\left((t-s)/\varepsilon\right) \, ds \in L_p\left(\Omega_T\right),
\]
and $D v^{(\varepsilon)} = (Dv)^{(\varepsilon)} \in L_p\left(\Omega_T\right)$.
Hence, $v^{(\varepsilon)} \in W_p^{1,1}\left(\Omega_T\right)$.
Moreover, from \eqref{eq0718_02} it follows that
\begin{equation}
							\label{eq0920_01}
v^{(\varepsilon)}_t - D_i((a_{ij}D_jv)^{(\varepsilon)})+ \lambda v^{(\varepsilon)} =0
\end{equation}
in $\Omega_T$ with the conormal derivative boundary condition on $(-\infty,T) \times \partial\Omega$.

We now consider $|v^{(\varepsilon)}|^{p-2}v^{(\varepsilon)}$, which satisfies
\[
|v^{(\varepsilon)}|^{p-2}v^{(\varepsilon)} \in L_q\left(\Omega_T\right),
\]
\[
\partial_t \left(|v^{(\varepsilon)}|^{p-2}v^{(\varepsilon)}\right) = (p-1)|v^{(\varepsilon)}|^{p-2} \partial_t v^{(\varepsilon)} \in L_q\left(\Omega_T\right),
\]
\[
D \left(|v^{(\varepsilon)}|^{p-2}v^{(\varepsilon)}\right) = (p-1)|v^{(\varepsilon)}|^{p-2} D v^{(\varepsilon)} \in L_q\left(\Omega_T\right),
\]
where $1/p+1/q=1$.
Thus, $|v^{(\varepsilon)}|^{p-2}v^{(\varepsilon)} \in W_q^{1,1}\left(\Omega_T\right)$.
Then, by applying $|v^{(\varepsilon)}|^{p-2}v^{(\varepsilon)}$ as (a kind of) test function to \eqref{eq0920_01} with the conormal derivative boundary condition, we get
\[
\frac{1}{p} \int_\Omega |v^{(\varepsilon)}(T,x)|^p \, dx + (p-1)\int_{\Omega_T} \left(a_{ij}D_jv\right)^{(\varepsilon)} |v^{(\varepsilon)}|^{p-2} D_i v^{(\varepsilon)} \, dx \, dt
\]
\[
+ \int_{\Omega_T} \lambda |v^{(\varepsilon)}|^p \, dx \, dt = 0.
\]
In particular,
\[
\int_{\Omega_T} \left( (p-1)\left(a_{ij}D_jv\right)^{(\varepsilon)} |v^{(\varepsilon)}|^{p-2} D_i v^{(\varepsilon)} + \lambda |v^{(\varepsilon)}|^p\right) \, dx \, dt \leq 0,
\]
where we know that $\left(a_{ij}D_jv\right)^{(\varepsilon)} \to a_{ij}D_jv$, $v^{(\varepsilon)} \to v$, $D_iv^{(\varepsilon)} = (D_iv)^{(\varepsilon)} \to D_iv$ in $L_p\left(\Omega_T\right)$ as $\varepsilon \to 0$.
Hence, by letting $\varepsilon \to 0$ in the above inequality and using the ellipticity condition, we get
\[
0 \geq \int_{\Omega_T} \left( (p-1)a_{ij}D_jv |v|^{p-2} D_iv + \lambda |v|^p \right) \, dx \, dt
\]
\[
\geq \int_{\Omega_T} \left( (p-1)\delta|v|^{p-2}|Dv|^2 + \lambda |v|^p \right) \, dx \, dt \geq 0.
\]
This implies that $v \equiv 0$ in $\Omega_T$.
\end{proof}
We now proceed to prove Theorem \ref{conormal -infty to T}.
\begin{proof}[Proof of Theorem \ref{conormal -infty to T}]
We first treat the case $p=q$ with $\omega \equiv 1$, and subsequently address the general space $L_{p,q,\omega}$.

We consider the case $p\in [2,\infty)$. Fix $\lambda > \lambda_0$, where $\lambda_0$ is from Theorem \ref{conormal entire time}. We begin by establishing the a priori estimate \eqref{eq0906_01}.
Let $u\in \cH_{p}^1(\Omega_T)$ be a solution to the equation
\begin{equation}
							\label{eq0908_02}
u_t - D_i(a_{ij}D_ju) + \lambda u = D_i g_i + f
\end{equation}
in $\Omega_T$ with the conormal derivative boundary condition on $(-\infty,T) \times \partial\Omega$. 
Set
\begin{equation}
							\label{eq0813_02}
\bar{g}_i = \left\{
\begin{aligned}
g_i \quad &\text{for} \quad -\infty < t < T,
\\
0 \quad &\text{for} \quad  t \geq T,
\end{aligned}
\right.
\quad
\bar{f} = \left\{
\begin{aligned}
f \quad &\text{for} \quad -\infty < t < T,
\\
0 \quad &\text{for} \quad  t \geq T.
\end{aligned}
\right.
\end{equation}
By Theorem \ref{conormal entire time}, there exists $\bar u \in H_p^{1/2,1}(\cQ)$, where $\cQ = \bR \times \Omega$, satisfying
\[
\bar u_t - D_i({a}_{ij} D_j \bar u) + \lambda \bar u = D_i \bar{g}_i + \bar{f}
\]
in $\cQ$ with the conormal derivative boundary condition on $\bR \times \partial\Omega$.
From the above equation, we deduce that 
$\|\bar u_t\|_{\bH_p^{-1}(\cQ)}$ is bounded, 
and hence $\bar u \in \cH_p^1(\cQ)$. 
This readily implies that $\bar u \in \cH_p^1(\Omega_T)$ as well.
We observe that $\bar u$ is also a solution to equation \eqref{eq0908_02} in $\Omega_T$ satisfying the conormal derivative boundary condition on $(-\infty,T)\times \partial \Omega$. By applying Lemma \ref{lem0906_2} to $\bar u - u$, we see that $u=\bar u$ in $\Omega_T$. Therefore, we have 
\[\|Du\|_{L_{p}(\Omega_T)}+\sqrt{\lambda}\|u\|_{L_{p}(\Omega_T)}
=\|D\bar u\|_{L_{p}(\Omega_T)}+\sqrt{\lambda}\|\bar u\|_{L_{p}(\Omega_T)}\]
\[\le \|D\bar u\|_{L_{p}(\cQ)}+\sqrt{\lambda}\|\bar u\|_{L_{p}(\cQ)}\le 
N \left( \|\bar g_i\|_{L_{p}(\cQ)}+\lambda^{-1/2}\|\bar f\|_{L_{p}(\cQ)}\right)
\]
\[
= N \left( \| g_i\|_{L_{p}(\Omega_T)}+\lambda^{-1/2}\|\bar f\|_{L_{p}(\Omega_T)}\right).
\]
We note that solvability follows immediately, 
since $\bar u$ is exactly the desired solution to equation \eqref{eq0908_02} in $\Omega_T$.

For $p \in (1,2)$, by the duality argument, we see that, as an a priori estimate, \eqref{eq0906_01} holds.
By extending $f$ and $g_i$ of the form \eqref{eq0813_02} and finding a solution in $\cQ$, the solvability is also established.

We now turn to the general space $L_{p,q,\omega}$.
By following the argument of the unweighted case, it is sufficient to show that there exist $\gamma_0$ and $\lambda_0$ such that, under Assumption \ref{assum rei coeffi} ($\gamma_0$) and $\lambda>\lambda_0$, if $u\in \cH_{p,q,\omega}^{1}(\Omega_T)$ satisfies 
\begin{equation}
    \label{eq0910_5}
u_t-D_i(a_{ij}D_ju)+\lambda u=0
\end{equation}
in $\Omega_T$ with the conormal derivative boundary condition on $(-\infty,T)\times \partial \Omega$, then $u\equiv 0$.

We choose $p_0>1$ sufficiently close to $1$ so that the small positive constants $\varepsilon_1$ and $\varepsilon_2$, defined by 
\[
\frac{1}{1+\varepsilon_1}=\frac{p-p_0}{p_0(p-1)}, \quad \frac{1}{1+\varepsilon_2}=\frac{q-p_0}{p_0(q-1)},
\]
satisfy $(\omega_1')^{{1+\varepsilon_1}}\in A_{p'}(\bR^d)$ and $(\omega_2')^{{1+\varepsilon_2}}\in A_{q'}(\bR)$, where $\omega_1'=\omega_1^{\frac{1}{1-p}}$, $\omega_2'=\omega_2^{\frac{1}{1-q}}$, and $p'$, $q'$ denote the conjugate exponents of $p$ and $q$, respectively.  From the doubling property of $A_p$-weights, we have a positive constant $N_1$ such that for any $r\in (0,\infty)$, 
 \begin{equation}
     \label{eq0910_7}
 \int_{B_{2r}} \omega_1^{1+\varepsilon_1}\,dx \le N_1 \int_{B_r} \omega_1^{1+\varepsilon_1}\,dx, \quad
 \int_{-4r^2}^{4r^2} \omega_2^{1+\varepsilon_2}\,dt \le N_1
 \int_{-r^2}^{r^2} \omega_2^{1+\varepsilon_2}\,dt. 
 \end{equation}
 By H\"older's inequality, for any $r>0$, we have 
\begin{equation}
\label{eq0827_1}
\|f\|_{L_{p_0}(Q_r\cap \Omega_T)}\le \|\omega_1'\|_{L_{1+\varepsilon_1}(B_r)}^{\frac{p-1}{p_0}}\|\omega_2'\|_{L_{1+\varepsilon_2}((-r^2,r^2))}^{\frac{q-1}{p_0}}\cdot \|f\|_{L_{p,q,\omega}(Q_r\cap \Omega_T)}.    
\end{equation}
Since $(\omega_1')^{{1+\varepsilon_1}}$ and $(\omega_2')^{{1+\varepsilon_2}}$ are locally integrable, the right-hand side of the above inequality is finite whenever $f\in L_{p,q,\omega}(\Omega_T)$.

For positive integer $n$, let $\zeta_n\in C_0^\infty(\bR^{d+1})$ be a cut-off function defined by 
\[\zeta_n(t,x) =
\begin{cases}
1, & (t,x)\in Q_{n} ,\\
0,  & (t,x)\in \bR^{d+1}\setminus Q_{2n}.
\end{cases}\]
Set $u_n=u\zeta_n$, where $u\in \cH_{p,q,\omega}^1(\Omega_T)$ is a solution to the equation \eqref{eq0910_5} in $\Omega_T$ with the conormal derivative boundary condition on $(-\infty,T)\times \partial \Omega$.
We observe that $u_n\in \cH_{p_0}^1(\Omega_T)$ from \eqref{eq0827_1} and $u_n$ satisfies
\begin{equation*}
(u_n)_t-D_i(a_{ij}D_j u_n)+\lambda u_n=D_ig_i^n+f^n    
\end{equation*}
in $\Omega_T$ with the conormal derivative boundary condition on $(-\infty,T)\times \partial \Omega$. Here, 
\[g_i^n=-a_{ij} uD_j \zeta_n, \quad f^n=-a_{ij}D_j uD_i \zeta_n +u\partial_t\zeta_n.
\]
Define $\overline{g_i^n}$ and $\overline{f^n}$ as in \eqref{eq0813_02}.
From \eqref{eq0827_1}, it follows that $\overline{g_i^n}$ and $\overline{f^n}$ belong to $ L_{p,q,\omega}(\cQ)\cap L_{p_0}(\cQ)$.
We now choose $\gamma_0$ and $\lambda_0$ from Theorem \ref{conormal entire time}, 
and then find $w_n \in H_{p,q,\omega}^{1/2,1}(\cQ)$ satisfying
\begin{equation}
\label{eq0910_2}
(w_n)_t - D_i(a_{ij}D_j w_n) + \lambda w_n = D_i \overline{g_i^n} + \overline{f^n}
\end{equation}
in $\cQ$ with the conormal derivative boundary condition on $\mathbb{R}\times \partial \Omega$.

We claim that $w_n \in \cH_{p_0}^1(\Omega_T)$. To prove the claim, we follow the argument used to verify solvability in weighted spaces in the proof of Theorem \ref{conormal entire time}. Indeed, using \eqref{eq0827_1}, we have 
\begin{equation}
    \label{eq0910_3}
    \||w_n|+|Dw_n|\|_{L_{p_0}(Q_{4n})}<\infty.
\end{equation}
For a nonnegative integer $k$, we take a sequence of smooth functions $\eta_k$ such that $\eta_k=0$ in $Q_{2^kn}$ and $1$ outside $Q_{2^{k+1}n}$, and 
\[
|D\eta_k|\le N2^{-k}, \quad |(\eta_k)_t|\le N2^{-2k}.
\]
We observe that $w_n\eta_k\in H_{p,q,\omega}^{1/2,1}(\cQ)$ satisfies 
\[
(w_n\eta_k)_t-D_i(a_{ij}D_j(w_n\eta_k))+\lambda w_n\eta_k
=
w_n(\eta_k)_t-a_{ij}D_jw_nD_i\eta_k-D_i(a_{ij}w_nD_j\eta_k)
\]
in $\cQ$ with the conormal derivative boundary condition on $\bR\times \partial\Omega$. 
From the a priori estimate for $w_n\eta_k$, we have
\[
\| |w_n\eta_k|+|D(w_n\eta_k)|\|_{L_{p,q,\omega}(\cQ)}
\le N \|w_n (\eta_k)_t\|_{L_{p,q,\omega}(\cQ)}
+N\|w_n D(\eta_k)\|_{L_{p,q,\omega}(\cQ)}
\]
\[+N\|Dw_n D(\eta_k)\|_{L_{p,q,\omega}(\cQ)},\]
which yields 
\[
 \||w_n|+|Dw_n|\|_{L_{p,q,\omega}(\cQ\cap (Q_{2^{k+2}n}\setminus Q_{2^{k+1}n}))}\le
\| |w_n|+|Dw_n|\|_{L_{p,q,\omega}(\cQ\setminus Q_{2^{k+1}n})}
\]
\[\le N2^{-k} \||w_n|+|Dw_n|\|_{L_{p,q,\omega}(\cQ\cap (Q_{2^{k+1}n}\setminus Q_{2^kn}))}.
\]
By induction, we see that for $k=1,2,\ldots$, 
\begin{equation}
    \label{eq0910_4}
     \||w_n|+|Dw_n|\|_{L_{p,q,\omega}(\cQ\cap (Q_{2^{k+1}n}\setminus Q_{2^{k}n}))}\le N 2^{-\frac{k(k-1)}{2}}
\| |u|+|Du|\|_{L_{p,q,\omega}(\cQ\cap Q_{4n})}.
\end{equation}
By H\"older's inequality, \eqref{eq0910_7}, and \eqref{eq0910_4}, it follows that for $k=1,2,\ldots$, 
\[
\||w_n|+|Dw_n|\|_{L_{p_0}(\cQ\cap (Q_{2^{k+1}n}\setminus Q_{2^{k}n}))}
\]
\[
\le \|\omega_1\|_{L_{1+\varepsilon_1}(B_{2^{k+1}n})}^{\frac{p-1}{p_0}} \|\omega_2\|_{L_{1+\varepsilon_2}(-4^{k+1}n^2,4^{k+1}n^2)}^{\frac{q-1}{p_0}}\cdot\||w_n|+|Dw_n|\|_{L_{p,q,\omega}(\cQ\cap (Q_{2^{k+1}n}\setminus Q_{2^{k}n}))}
\]
\[
\le NN_1^{k(\frac{p+q-2}{p_0})}2^{-\frac{k(k-1)}{2}}
\| |w_n|+|Dw_n|\|_{L_{p_1}(Q_{4n})}.
\]
This estimate together with \eqref{eq0910_3} implies that $w_n,Dw_n\in L_{p_0}(\cQ)$. Moreover, from the equation \eqref{eq0910_2}, we see that $\|(w_n)_t\|_{\bH_{p_0}^{-1}(\cQ)}<\infty$. Hence, we obtain that $w_n\in \cH_{p_0}^1(\cQ)$ and this directly implies that $w_n\in \cH_{p_0}^1(\Omega_T)$. 

Since $w_n$ satisfies the same equation as $u_n$ in $\Omega_T$, 
applying the $L_{p_0}(\Omega_T)$-estimate, previously established in the case $p=q$ and $\omega \equiv 1$, to $u_n - w_n$ yields $u_n = w_n$ in $\Omega_T$.
Due to the fact that $\overline{g_i^n}, \overline{f^n}\to 0$ in $ L_{p,q,\omega}(\cQ)$, we obtain $w_n\to 0$ in $H_{p,q,\omega}^{1/2,1}(\cQ)$. Combining these with the observation that $u_n\to u$ in $\cH_{p,q,\omega}^1(\Omega_T)$, we conclude that $u\equiv 0$.
The theorem is proved.
\end{proof}

Finally, we deduce Corollary \ref{conormal bounded interval} from Theorem \ref{conormal -infty to T}.

\begin{proof}[Proof of Corollary \ref{conormal bounded interval}]
We first establish the estimate \eqref{eq0910_9}. We choose $\gamma_0$ and $\lambda_0$ in Theorem \ref{conormal -infty to T} and fix $\lambda>\lambda_0$. 
 Let $u\in \cH_{p,q,\omega,0}^{1}((0,T)\times \Omega)$ be a solution to 
        \begin{equation*}
    u_t-D_i(a_{ij}D_j u)=D_ig_i+f 
    \end{equation*}
    in $(0,T)\times \Omega$ with the conormal derivative boundary condition on $(0,T)\times \partial\Omega$.
    We extend $u$ by zero for $t\le0$, denoting the extension by $\bar u$. This extended function $\bar u\in \cH_{p,q,\omega}^{1}(\Omega_T)$ then solves  
        \begin{equation*}
    \bar u_t-D_i(a_{ij}D_j \bar  u)=D_i\bar g_i+ \bar f
    \end{equation*}
    in $\Omega_T$ with the conormal derivative boundary condition on $(-\infty,T)\times \partial\Omega$, where $\bar g_i=g_iI_{(0,T)}$ and $\bar f=fI_{(0,T)}$. 
    Now, set $w=e^{-\lambda t}\bar u$, where $\lambda>\lambda_0$. A direct calculation shows that $w$ satisfies 
    \begin{equation}
        \label{0911_1}
w_t - D_i(a_{ij}D_jw) + \lambda w = D_i(e^{-\lambda t}\bar g_i) + e^{-\lambda t} \bar f
    \end{equation}
in $\Omega_T$ with the conformal derivative condition on $(-\infty,T) \times \partial\Omega$. By applying the estimate \eqref{eq0906_01} to $w$, we obtain \eqref{eq0910_9}.

The solvability follows by reversing the steps: 
first construct $w$ satisfying \eqref{0911_1}, 
then set $\bar u = e^{\lambda t} w$, 
and finally define $u = \bar{u}|_{(0,T) \times \Omega}$. 
It then follows that $u$ solves the desired equation and $u\in \cH_{p,q,\omega}^1((0,T)\times \Omega)$.
Moreover, we claim that $u \in \cH_{p,q,\omega,0}^1((0,T)\times \Omega)$. 
To see this, it suffices to show that $w=0$ on $(-\infty,0)\times \Omega$. 
In fact, $w$ satisfies
\[
w_t - D_i(a_{ij}D_j w) + \lambda w = 0
\]
in $(-\infty,0)\times \Omega$ with the conormal derivative boundary condition on $(-\infty,0)\times \partial\Omega$. 
Applying Theorem \ref{conormal -infty to T}, we deduce that $w=0$ throughout $(-\infty,0)\times \Omega$. 
Consequently, $u$ belongs to $\cH_{p,q,\omega,0}^1((0,T)\times \Omega)$.
 The corollary is proved.
\end{proof}

\begin{thebibliography}{10}

\bibitem{MR1911202}
P.~Auscher and M.~Qafsaoui.
\newblock Observations on {$W^{1,p}$} estimates for divergence elliptic
equations with {VMO} coefficients.
\newblock {\em Boll. Unione Mat. Ital. Sez. B Artic. Ric. Mat. (8)},
5(2):487--509, 2002.

\bibitem{MR1239929}
Marco Bramanti and M.~Cristina Cerutti.
\newblock {$W_p^{1,2}$} solvability for the {C}auchy-{D}irichlet problem for
parabolic equations with {VMO} coefficients.
\newblock {\em Comm. Partial Differential Equations}, 18(9-10):1735--1763,
1993.

\bibitem{MR2110431}
Sun-Sig Byun.
\newblock Elliptic equations with {BMO} coefficients in {L}ipschitz domains.
\newblock {\em Trans. Amer. Math. Soc.}, 357(3):1025--1046, 2005.

\bibitem{MR3467697}
Sun-Sig Byun, Dian~K. Palagachev, and Lubomira~G. Softova.
\newblock Global gradient estimates in weighted {L}ebesgue spaces for parabolic
operators.
\newblock {\em Ann. Acad. Sci. Fenn. Math.}, 41(1):67--83, 2016.

\bibitem{MR2069724}
Sun-Sig Byun and Lihe Wang.
\newblock Elliptic equations with {BMO} coefficients in {R}eifenberg domains.
\newblock {\em Comm. Pure Appl. Math.}, 57(10):1283--1310, 2004.

\bibitem{MR2107043}
Sun-Sig Byun and Lihe Wang.
\newblock The conormal derivative problem for elliptic equations with {BMO}
coefficients on {R}eifenberg flat domains.
\newblock {\em Proc. London Math. Soc. (3)}, 90(1):245--272, 2005.

\bibitem{MR2139880}
Sun-Sig Byun and Lihe Wang.
\newblock {$L^p$} estimates for parabolic equations in {R}eifenberg domains.
\newblock {\em J. Funct. Anal.}, 223(1):44--85, 2005.

\bibitem{MR1486629}
L.~A. Caffarelli and I.~Peral.
\newblock On {$W^{1,p}$} estimates for elliptic equations in divergence form.
\newblock {\em Comm. Pure Appl. Math.}, 51(1):1--21, 1998.

\bibitem{MR1191890}
Filippo Chiarenza, Michele Frasca, and Placido Longo.
\newblock Interior {$W^{2,p}$} estimates for nondivergence elliptic equations
with discontinuous coefficients.
\newblock {\em Ricerche Mat.}, 40(1):149--168, 1991.

\bibitem{MR1088476}
Filippo Chiarenza, Michele Frasca, and Placido Longo.
\newblock {$W^{2,p}$}-solvability of the {D}irichlet problem for nondivergence
elliptic equations with {VMO} coefficients.
\newblock {\em Trans. Amer. Math. Soc.}, 336(2):841--853, 1993.

\bibitem{MR4387198}
Jongkeun Choi, Hongjie Dong, and Zongyuan Li.
\newblock Optimal regularity of mixed {D}irichlet-conormal boundary value
problems for parabolic operators.
\newblock {\em SIAM J. Math. Anal.}, 54(2):1393--1427, 2022.

\bibitem{MR3947859}
Jongkeun Choi and Doyoon Kim.
\newblock Weighted {$L_{p,q}$}-estimates for higher order elliptic and
parabolic systems with {${\rm BMO}_x$} coefficients on {R}eifenberg flat
domains.
\newblock {\em Calc. Var. Partial Differential Equations}, 58(3):Paper No. 90,
29, 2019.

\bibitem{MR1405255}
G.~Di~Fazio.
\newblock {$L^p$} estimates for divergence form elliptic equations with
discontinuous coefficients.
\newblock {\em Boll. Un. Mat. Ital. A (7)}, 10(2):409--420, 1996.

\bibitem{MR2835999}
Hongjie Dong and Doyoon Kim.
\newblock Higher order elliptic and parabolic systems with variably partially
{BMO} coefficients in regular and irregular domains.
\newblock {\em J. Funct. Anal.}, 261(11):3279--3327, 2011.

\bibitem{MR2764911}
Hongjie Dong and Doyoon Kim.
\newblock {$L_p$} solvability of divergence type parabolic and elliptic systems
with partially {BMO} coefficients.
\newblock {\em Calc. Var. Partial Differential Equations}, 40(3-4):357--389,
2011.

\bibitem{MR2771670}
Hongjie Dong and Doyoon Kim.
\newblock On the {$L_p$}-solvability of higher order parabolic and elliptic
systems with {BMO} coefficients.
\newblock {\em Arch. Ration. Mech. Anal.}, 199(3):889--941, 2011.

\bibitem{MR3013054}
Hongjie Dong and Doyoon Kim.
\newblock The conormal derivative problem for higher order elliptic systems
with irregular coefficients.
\newblock In {\em Recent advances in harmonic analysis and partial differential
	equations}, volume 581 of {\em Contemp. Math.}, pages 69--97. Amer. Math.
Soc., Providence, RI, 2012.

\bibitem{MR3812104}
Hongjie Dong and Doyoon Kim.
\newblock On {$L_p$}-estimates for elliptic and parabolic equations with
{$A_p$} weights.
\newblock {\em Trans. Amer. Math. Soc.}, 370(7):5081--5130, 2018.

\bibitem{MR4030286}
Hongjie Dong and Doyoon Kim.
\newblock {$L_p$}-estimates for time fractional parabolic equations in
divergence form with measurable coefficients.
\newblock {\em J. Funct. Anal.}, 278(3):108338, 66, 2020.

\bibitem{MR4345837}
Hongjie Dong and Doyoon Kim.
\newblock Time fractional parabolic equations with measurable coefficients and
embeddings for fractional parabolic {S}obolev spaces.
\newblock {\em Int. Math. Res. Not. IMRN}, (22):17563--17610, 2021.

\bibitem{MR4541086}
Hongjie Dong and Yanze Liu.
\newblock Sobolev estimates for fractional parabolic equations with space-time
non-local operators.
\newblock {\em Calc. Var. Partial Differential Equations}, 62(3):Paper No. 96,
49, 2023.

\bibitem{MR807149}
Jos\'e{} Garc\'ia-Cuerva and Jos\'e{}~L. Rubio~de Francia.
\newblock {\em Weighted norm inequalities and related topics}, volume 116 of
{\em North-Holland Mathematics Studies}.
\newblock North-Holland Publishing Co., Amsterdam, 1985.
\newblock Notas de Matem\'atica, 104. [Mathematical Notes].

\bibitem{MR312139}
Richard Hunt, Benjamin Muckenhoupt, and Richard Wheeden.
\newblock Weighted norm inequalities for the conjugate function and {H}ilbert
transform.
\newblock {\em Trans. Amer. Math. Soc.}, 176:227--251, 1973.

\bibitem{MR4920684}
Pilgyu Jung and Doyoon Kim.
\newblock {$L_p$}-estimates for parabolic equations in divergence form with a
half-time derivative.
\newblock {\em J. Differential Equations}, 443:Paper No. 113560, 41, 2025.

\bibitem{MR4387945}
Doyoon Kim, Seungjin Ryu, and Kwan Woo.
\newblock Parabolic equations with unbounded lower-order coefficients in
{S}obolev spaces with mixed norms.
\newblock {\em J. Evol. Equ.}, 22(1):Paper No. 9, 40, 2022.

\bibitem{MR2304157}
N.~V. Krylov.
\newblock Parabolic and elliptic equations with {VMO} coefficients.
\newblock {\em Comm. Partial Differential Equations}, 32(1-3):453--475, 2007.

\bibitem{MR2435520}
N.~V. Krylov.
\newblock {\em Lectures on elliptic and parabolic equations in {S}obolev
	spaces}, volume~96 of {\em Graduate Studies in Mathematics}.
\newblock American Mathematical Society, Providence, RI, 2008.

\bibitem{MR563790}
N.~V. Krylov and M.~V. Safonov.
\newblock A property of the solutions of parabolic equations with measurable
coefficients.
\newblock {\em Izv. Akad. Nauk SSSR Ser. Mat.}, 44(1):161--175, 239, 1980.

\bibitem{MR4340793}
Emma-Karoliina Kurki and Carlos Mudarra.
\newblock On the extension of {M}uckenhoupt weights in metric spaces.
\newblock {\em Nonlinear Anal.}, 215:Paper No. 112671, 20, 2022.

\bibitem{MR3186805}
Antoine Lemenant, Emmanouil Milakis, and Laura~V. Spinolo.
\newblock On the extension property of {R}eifenberg-flat domains.
\newblock {\em Ann. Acad. Sci. Fenn. Math.}, 39(1):51--71, 2014.

\bibitem{MR1774162}
Bengt~Ove Turesson.
\newblock {\em Nonlinear potential theory and weighted {S}obolev spaces},
volume 1736 of {\em Lecture Notes in Mathematics}.
\newblock Springer-Verlag, Berlin, 2000.

\end{thebibliography}
\end{document}